\documentclass[11pt, a4paper]{article}
\usepackage{geometry}
\usepackage[T1]{fontenc}
\usepackage{lmodern}
\usepackage{textcomp}
\usepackage[utf8]{inputenc}
\usepackage[english]{babel}
\usepackage{amsmath}
\usepackage{amssymb}
\usepackage[parfill]{parskip}
\usepackage[shortlabels]{enumitem}
\usepackage{amsthm}
\begingroup
    \makeatletter
    \@for\theoremstyle:=definition,remark,plain\do{%
        \expandafter\g@addto@macro\csname th@\theoremstyle\endcsname{%
            \addtolength\thm@preskip\parskip
            }%
        }
\endgroup

\usepackage{amsthm}
\usepackage{mathabx}
\usepackage{enumitem}
\usepackage{bm}
\usepackage{booktabs}
\usepackage{subfig}
\usepackage{placeins}
\usepackage{graphicx}
\usepackage{sidecap}
\usepackage[bookmarksdepth=1]{hyperref}
\usepackage[numbers]{natbib}
\usepackage{framed}
\usepackage{multirow}
\usepackage{changepage}
\usepackage{datetime}
\usepackage{xcolor}
\usepackage{accents}

\newdateformat{mydate}{\monthname[\THEMONTH], \THEYEAR}
\newcommand{\foralls}{\forall \,}

\newcommand{\inner}[2]{\left(#1,#2\right)}
\newcommand{\uv}[1]{\underline{u}_{#1}}
\newcommand{\wv}[1]{\underline{w}_{#1}}
\newcommand{\Smo}[1]{\tau_{#1} L^{-1}_{#1}}
\newcommand{\hmax}{h}
\newcommand{\hmaxh}[1]{h_{#1}}

\newtheorem{theorem}{Theorem}[section]
\newtheorem{algorithms}{Algorithm}[section]
\newtheorem{corollary}{Corollary}[section]
\newtheorem{lemma}{Lemma}[section]

\newtheorem{remark}{Remark}[section]
\newtheorem{notation}{Notation}[section]

\numberwithin{equation}{section}

\title{\vspace{-15mm}\fontsize{18pt}{10pt}\selectfont\textbf{Multigrid solvers for isogeometric discretizations of the second biharmonic problem}}

\author{
\large
Jarle Sogn$^1$ and Stefan Takacs$^2$
\\[2mm] 
\normalsize $^1$ Johann Radon Institute for Computational and Applied Mathematics, \\
\normalsize Austrian Academy of Sciences, \\
\normalsize Altenberger Str. 69, 4040 Linz, Austria \\
\normalsize $^2$ Institute of Computational Mathematics, \\ 
\normalsize Johannes Kepler University Linz, \\
\normalsize Altenberger Str. 69, 4040 Linz, Austria 
\vspace{-5mm}
}
\date{}

\begin{document}

\maketitle

\begin{abstract}
  We develop a multigrid solver for the second biharmonic problem
  in the context of Isogeometric Analysis (IgA), where we also allow a zero-order
  term. In a previous paper, the authors
  have developed an analysis for the first biharmonic problem based on Hackbusch's
  framework. This analysis can only be extended to the second biharmonic
  problem if one assumes uniform grids.
  In this paper, we prove a multigrid convergence estimate using Bramble's
  framework for multigrid
  analysis without regularity assumptions. We show that the bound for the convergence
  rate is independent of the scaling of the zero-order term and the spline degree. It
  only depends linearly on the number of levels, thus
  logarithmically on the grid size.  Numerical experiments are provided which
  illustrate the convergence theory and the  efficiency of the proposed multigrid approaches.
\end{abstract}

\section{Introduction}
We consider multigrid methods for biharmonic problems discretized by Isogeometric Analysis (IgA). In particular, we consider the following model problem:
Given a bounded domain $\Omega\subset \mathbb R^d$, $d\in\{2,3\}$, with Lipschitz boundary $\partial\Omega$,
a parameter $\beta \geq 0$ and sufficiently smooth functions $f$, $g_1$, and $g_2$,
find a function $u$ such that
\begin{align}
  \label{eq:probStrong}
  \begin{split}
    \beta u + \Delta^2 u  &= f \quad \text{in} \quad \Omega,\\
    u  &= g_1 \quad \text{on} \quad \partial\Omega,\\
    \Delta u  &= g_2 \quad \text{on} \quad \partial\Omega
  \end{split}
\end{align}
holds in a variational sense.
For $\beta = 0$, this problem is known as the \textit{second biharmonic problem}, which is of interest for plate theory (cf. \cite{ciarlet2002finite}) and Stokes streamline equations (cf. \cite{girault2012finite}). Problems with $\beta > 0$ are of particular interest in the context
of optimal control problems, where the constraint is a second order elliptic operator. The
optimality systems associated to these optimal control problems can be preconditioned robustly using preconditioners that rely on solving~\eqref{eq:probStrong}, see \cite{MarNieNor17,SogZul18,beigl2019robust,mardal2020robust}. The problem~\eqref{eq:probStrong}
is obtained when considering the full observation; if one considers an optimal
control problem with limited observation, one would obtain a similar problem, where the mass term $\beta u$ is multiplied with the characteristic function for
the observation domain.

We derive a standard variational formulation of the model problem, which lives in the Sobolev
space $H^2(\Omega)$. For the discretization, we use Isogeometric Analysis (IgA) since
it easily allows for $H^2$-conforming discretizations. Particularly, we consider a
discretization based on tensor product B-splines of some degree $p>1$ and maximum smoothness, i.e.,
$p-1$ times continuously differentiable. For the derivation of the multigrid solver,
we set up a hierarchy of grids as obtained by uniform refinement.
Since we keep spline degree and spline smoothness fixed, we obtain nested spaces.

Concerning the choice of the smoother, there are many possibilities. We are interested in a smoother
that yields a $p$-robust multigrid method. The first $p$-robust multigrid solvers were based on the 
{boundary corrected mass smoother} \cite{hofreither2017robust} and
the {subspace corrected mass smoother} \cite{hofreither2016robust}. Both have been
formulated for the Poisson problem. Since the subspace corrected mass smoother is more flexible and
has proven itself more efficient in practice, we restrict ourselves to that smoother. The multigrid
solvers with subspace corrected mass smoother have been extended to the first biharmonic problem in \cite{sogn2019robust} and to the second and third biharmonic problem in the thesis \cite{sogn2018schur}. The convergence estimates are shown using the standard splitting of the analysis into approximation
property and smoothing property, as proposed by Hackbusch, cf.~\cite{hackbusch2013multi}.

The theory in all of these papers requires that the grids are uniform since they have been based on the
$p$-robust approximation error estimates from~\cite{takacs2016approximation}, which are valid only
in this case. Since then, newer $p$-robust approximation error estimates,
see \cite{sande2020explicit, sande2019sharp}, have been proposed, which do not require uniform
grids. Using these new estimates, it is straightforward to relax this assumption and to show analogous results for the Poisson problem as well as the first biharmonic problem for quasi uniform grids. However, this is not straightforward for the second biharmonic problem, since the proof requires a certain commutativity property
(cf.~\cite[Lemma~9.2]{sogn2018schur}), which is only valid in case of uniform grids.

In this paper, we go another way. We base the analysis on the framework introduced by
Bramble et al., cf.~\cite{bramble1991convergence,bramble2018multigrid}. This allows us to drop the requirement that the grids are
uniform. While this analysis could also be performed for other kinds of boundary conditions, like
the first biharmonic problem, we restrict ourselves to the second biharmonic problem since it
has previously turned out to be the more challenging one.
For this setting, we prove a multigrid convergence estimate which is robust with respect to
the spline degree $p$ and which only depends logarithmically on the grid size $h$.

Moreover, we show that the convergence is robust in the parameter $\beta\ge0$. This analysis is motivated by the
mentioned optimal control problem. Such parameter-robust multigrid solvers are also known for
the Poisson problem, see \cite{olshanskii2000convergence} for an analysis based on Hackbusch's
framework. There, the authors also provide a regularity result for the corresponding partial
differential equation (PDE), which is based on standard results for the Poisson problem. In
our case, we do not need to do that since Bramble's analysis is not based on any regularity
assumptions.

In the numerical experiments, one can observe that the convergence of a multigrid solver
with subspace corrected mass smoother degrades if the geometry gets distorted. While this is
also true for the Poisson problem, this dependence is significantly amplified for the
biharmonic problem. The reason for the geometry dependence of the convergence rates is that
the subspace corrected mass smoother is based on the tensor product structure of the spline space.
This tensor product structure is distorted by the geometry mapping. So, the contributions of
the geometry function are ignored when setting up the smoother. We aim to overcome this problem
by considering a hybrid smoother that combines the proposed smoother with Gauss-Seidel sweeps,
see also~\cite{sogn2019robust,sogn2018schur}.

Alternative smoothers based on overlapping multiplicative Schwarz techniques have been
considered in \cite{de2020robust,mardal2020robust}. Both approaches give good
numerical results for the biharmonic problem. However, there is no rigorous, $p$-robust convergence
theory available for these methods. It is worth mentioning that, as an 
alternative for solving biharmonic problems on the primal form, various kinds of mixed or
non-conforming formulations have been
developed, cf. \cite{doi:10.1137/0726062, Zhang:Xu, hanisch1993multigrid,
rafetseder2018decomposition, chen2015multigrid}.

The remainder of the paper is organized as follows. We introduce IgA, the biharmonic model problem
in its variational form and its discretization in Section~\ref{sec:prelims}. In Section~\ref{sec:MG}, the multigrid method is introduced and we state sufficient conditions for its convergence.  We develop the approximation error estimates needed for the convergence estimates in Section~\ref{Approx}. The choice of the smoother, the smoothing properties and the resulting multigrid convergence results are addressed in Section~\ref{sec:smoothers}. Finally, we provide numerical results in Section~\ref{sec:numerical}.

\section{Model problem and its discretization}
\label{sec:prelims}

\subsection{The biharmonic model problem}

Following the usual design principles of IgA,
we assume that the computational domain $\Omega\subset \mathbb R^d$ 
has a Lipschitz boundary $\partial \Omega$ and that it is parameterized
by a geometry function
\begin{equation}
\nonumber
		\textbf G: \widehat \Omega =(0,1)^d\rightarrow \Omega
		=\textbf G(\widehat \Omega),
\end{equation}
whose third weak derivatives are almost everywhere uniformly bounded. The parameterization has the property 
\begin{equation}
\label{eq:GeoMapCond}
    \|\nabla^r \mathbf{G}\|_{L^\infty(\widehat \Omega)} \leq c_1 \quad \text{and} \quad \|\left(\nabla^r \mathbf{G}\right)^{-1}\|_{L^\infty(\widehat \Omega)} \leq c_2, \quad  \text{for}\quad r=1,2,3,
\end{equation}
for some constants $c_1$ and $c_2$.

After homogenization, the variational formulation of the model problem
\eqref{eq:probStrong} reads as follows.
Given $f\in L^2(\Omega)$ and $\beta\in\mathbb{R}$ with $\beta\geq 0$, find $u\in V:=H^2(\Omega)\cap H^1_0(\Omega)$ such that 
\begin{equation}\label{eq:problem1}
\beta(u, v)_{L^2(\Omega)}+(\Delta u, \Delta v)_{L^2(\Omega)}= (f, v)_{L^2(\Omega)}  \quad \foralls v \in V.
\end{equation}
Here and in what follows, $L^2(\Omega)$ and $H^r(\Omega)$ denote the standard Lebesgue and
Sobolev spaces with standard inner products $(\cdot,\cdot)_{L^2(\Omega)}$, $(\cdot,\cdot)_{H^r(\Omega)}$ and norms $\|\cdot \|_{L^2(\Omega)}$, $\|\cdot \|_{H^r(\Omega)}$. $H^1_0(\Omega)$ is the standard subspace of $H^1(\Omega)$ containing the functions with vanishing trace. 
On $V$, we define the bilinear form $\inner{\cdot}{\cdot}_{\mathcal{B}}$ via
\[
		\inner{u}{v}_{\mathcal{B}} := (\Delta u, \Delta v)_{L^2(\Omega)}
		\quad\foralls u,v\in V,
\]
which is an inner product since we have the Poincaré like inequality
\begin{equation}
  \label{eq:2FIne}
\|u\|_{H^2(\Omega)}\leq c_{\Omega}\|\Delta u\|_{L^2(\Omega)}
= c_{\Omega} \|u\|_{\mathcal B} \quad \foralls u\in V,
\end{equation}
where $c_{\Omega}$ is a constant that depends only on the shape
of $\Omega$, cf.~\cite{mardal2020robust}.

Using the substitution rule for integration and the chain rule for differentiation, \eqref{eq:problem1} can be expressed in terms of
integrals on the parameter domain $\widehat\Omega$. In IgA, this is usually
done in order to simplify the evaluation of the integrals using quadrature
rules. Besides these inner products, there are also standard inner products
for the parameter domain, like $(\cdot,\cdot)_{L^2(\widehat\Omega)}$
and $(\cdot,\cdot)_{\widehat{\mathcal{B}}}$, where the latter is given by
\[
		\inner{\widehat u}{\widehat v}_{\widehat{\mathcal{B}}}
		:= (\Delta \widehat u, \Delta \widehat v)_{L^2(\widehat \Omega)}
		\quad\foralls \widehat u,\widehat v\in \widehat V:=
		H^2(\widehat \Omega)\cap H^1_0(\widehat\Omega).
\]
Also for the parameter domain $\widehat\Omega$, the result~\eqref{eq:2FIne} holds. So, we know
\begin{equation}
\nonumber 
\|u\|_{H^2(\widehat \Omega)}\leq c_{\widehat \Omega}\|\Delta u\|_{L^2(\widehat\Omega)}
= c_{\widehat\Omega} \|u\|_{\widehat{\mathcal B}} \quad \foralls u\in \widehat V.
\end{equation}
We know (cf.~\cite{sogn2018schur})
that there exist constants $\underline{c}_M$, $\overline{c}_M$, 
$\underline{c}_B$ and $\overline{c}_B$ only depending on the constants $c_1$, $c_2$ and the shape of $\Omega$ such that
\begin{equation}
  \label{eq:geoEquiv}
  \begin{split}
  \underline{c}_M\, (u,u)_{L^2(\Omega)}
  &\leq (\widehat u,\widehat u)_{L^2(\widehat{\Omega})}
  \leq \overline{c}_M\, (u,u)_{L^2(\Omega)}
  \quad \text{and}\\
  \underline{c}_B\, (u,u)_{\mathcal B}
  & \leq (\widehat u,\widehat u)_{\widehat{\mathcal{B}}}
  \leq \overline{c}_B\, (u,u)_{\mathcal B}
  \end{split}
\end{equation}
for all $u\in V$ with $\widehat u = u\circ \bm{G}\in\widehat V$.
We define a simplified bilinear form  $\left(\cdot, \cdot\right)_{\bar{\mathcal{B}}}$ as the inner product obtained by removing the cross terms from the inner product $\left(\cdot, \cdot\right)_{\widehat{\mathcal{B}}}$, that is,
\begin{equation}
\nonumber
\left(\widehat u, \widehat v\right)_{\bar{\mathcal{B}}} := \sum^d_{k=1}\left(\partial_{x_kx_k} \widehat u , \partial_{x_kx_k} \widehat v \right)_{L^2(\widehat{\Omega})}
\quad\foralls \widehat u,\widehat v\in \widehat V. 
\end{equation}
Here and in what follows, $\partial_{x} := \frac{\partial}{\partial x}$ and $\partial_{xy} := \partial_{x} \partial_{y}$ and $\partial_{x}^r := \frac{\partial^r}{\partial x^r}$ denote partial derivatives.
The original bilinear form and the simplified bilinear form are spectrally equivalent, which implies that also the simplified bilinear form is an inner product.
\begin{lemma}
  \label{lemma:equivB}
  The inner products 
  $(\cdot,\cdot)_{\widehat{\mathcal{B}}}$
  and
  $(\cdot,\cdot)_{\bar{\mathcal{B}}}$ are spectrally
  equivalent, that is,
  \begin{equation}
    \nonumber
\left(\widehat u, \widehat u\right)_{\bar{\mathcal{B}}} \leq \left(\widehat u, \widehat u\right)_{\widehat{\mathcal{B}}} \leq d\left(\widehat u, \widehat u\right)_{\bar{\mathcal{B}}}
    \quad \foralls \widehat u\in \widehat V .
\end{equation} 
\end{lemma}
\begin{proof}
  From \cite{grisvard2011elliptic, grisvard1992singularities},
  it follows that $\|\Delta \widehat u\|_{L^2(\widehat\Omega)}
  = \|\nabla^2 \widehat v\|_{L^2(\widehat\Omega)}$
  for $\widehat u,\widehat v\in \widehat V$. Using this, we obtain
	\begin{align*}
	\|\widehat u\|_{\widehat{\mathcal{B}}}^2
		& = \|\Delta \widehat u\|_{L^2(\widehat\Omega)}^2
		= \|\nabla^2\widehat u\|_{L^2(\widehat\Omega)}^2
		 = \underbrace{
		 \sum_{k=1}^d
		 \|\partial_{x_kx_k}\widehat u\|_{L^2(\widehat\Omega)}^2
		 }_{\displaystyle =\|\widehat u\|_{\bar{\mathcal{B}}}^2}
		 + 
		 \underbrace{
		 \sum_{k=1}^d\sum_{l\in\{1,\ldots,d\}\backslash\{k\}}
		 \|\partial_{x_kx_l}\widehat u\|_{L^2(\widehat\Omega)}^2
		 }_{\displaystyle \ge 0},
	\end{align*}
	which shows the first side of the inequality.
	Using the Cauchy-Schwarz inequality
	and $ab \le \tfrac12( a^2+ b^2)$, we obtain
	\begin{align*}
	\|\widehat u\|_{\widehat{\mathcal{B}}}^2 &= \sum_{k=1}^d\sum_{l=1}^d \left(\partial^2_{x_k} \widehat u,\partial^2_{x_l} \widehat u\right)_{L^2(\widehat\Omega)}
		\le \frac{1}{2}  \sum_{k=1}^d\sum_{l=1}^d \left( \left\|\partial_{x_k}^2 \widehat u\right\|_{L^2(\widehat\Omega)}^2 + \left\|\partial_{x_l}^2 \widehat u\right\|_{L^2(\widehat\Omega)}^2\right)
		= d \|\widehat u\|_{\bar{\mathcal{B}}}^2,
	\end{align*}
	which shows second side of the inequality.
\end{proof}
\begin{remark}
A analogous result holds for the domain $\Omega$, which satisfies condition \eqref{eq:GeoMapCond}. In this case, the constants also depend on
the shape of $\Omega$.
\end{remark}

\subsection{Discretization}

We consider a discretization using tensor product B-splines in the context of IgA. We start by defining these splines on the parameter domain $\widehat \Omega$. 
Let $C^k(0,1)$ denote the space of all continuous functions mapping $(0,1)\rightarrow \mathbb{R}$ that are $k$ times continuously differentiable and let $\mathcal{P}_p$ be the space of polynomials of degree at most $p$.
For any sequence of grid points $\bm{\tau}:= (\tau_0,\ldots,\tau_{N+1})$ with
\[
0 = \tau_0 < \tau_1 <\cdots < \tau_N < \tau_{N+1} = 1,
\]
we define the space $S_{p,\bm\tau}$ of splines of degree $p$ with maximum smoothness by
\[
		S_{p,\bm\tau} := \left\lbrace v\in C^{p-1}(0,1) : v|_{(\tau_j,\tau_{j+1})}\in \mathcal{P}_p,\; j=0,1,\ldots,N \right\rbrace.
\]
The size of the largest and the smallest interval are denoted by
\begin{equation}
    \nonumber
		h_{\bm\tau} := \max_{j=0,\ldots,N}(\tau_{j+1}-\tau_j) 
		\quad\text{and}\quad
		h_{\bm\tau,\mathrm{min}} :=\min_{j=0,\ldots,N}(\tau_{j+1}-\tau_j),
\end{equation}
respectively. For the parameter domain, we define a spline space by
tensorization, which we transfer to the physical domain using the
pull-back principle, thus we define for given sequences of grid points 
$\bm{\tau}_{\ell,1},\ldots,\bm{\tau}_{\ell,d}$ the spaces
\[
		\widehat V_\ell := 
		\left(
		\bigotimes^d_{i=1} S_{p,\bm{\tau}_{\ell,i}}
		\right) \cap H^1_0(\widehat \Omega)\subset \widehat V
		\quad\text{and}\quad
		V_\ell := \{ f\circ \textbf G^{-1} : f\in \widehat V_\ell \}\subset V.
\]
Here and in what follows, the tensor product space $\bigotimes^d_{i=1} S_{p,\bm{\tau}_{\ell,i}}$ is the space of all linear combinations of functions of the form $v(x_1,\ldots,x_d)=v_1(x_1)\cdots v_d(x_d)$ with $v_i\in S_{p,\bm{\tau}_{\ell,i}}$. The spline degree $p$ could be different for each of the spacial directions. For notational convenience, we restrict ourselves to a uniform choice of the degree.

The corresponding minimum and maximum grid size are denoted by
\[
		h_\ell := \max_{i=1,\ldots,d} h_{\bm{\tau}_{\ell,i}}
		\quad\text{and}\quad
		h_{\ell,\mathrm{min}} := \min_{i=1,\ldots,d} h_{\bm{\tau}_{\ell,i},\mathrm{min}}.
\]
For the multigrid methods we, set up a sequence nested spline spaces
\[
		V_0 \subset V_1 \subset \cdots \subset V_{L} \subset V
		\quad\text{with}\quad h_0>h_1> \cdots >h_L>0
\]
based on a sequence of nested grids.

We assume that all grids are quasi uniform, that is, there is a constant $c_q$ such that
\begin{equation}
	  \label{eq:quasiuniform}
    h_\ell \leq c_q \,h_{\ell,\mathrm{min}} \quad \text{for} \quad \ell = 0,1,\ldots, L.
\end{equation}
We also assume that the ratio of the grid sizes of any two consecutive grids is bounded, that is, there is a constant $c_r$ such that
\begin{equation}
    \label{eq:assGrids}
    h_\ell \leq c_r\, h_{\ell-1}  \quad \text{for} \quad \ell = 1,\ldots, L.
\end{equation}
If the grids are obtained by uniform refinements of the coarsest grid, then this condition is naturally satisfied with $c_r=2$.

By applying a Galerkin discretization, we obtain the following discrete problem: Find $u_\ell \in V_\ell$ such that
\begin{equation}
  \label{eq:probDisc}
\beta (u_\ell,v_\ell)_{L^2(\Omega)}+ (u_\ell,v_\ell)_{\mathcal{B}} = (f, v_\ell)_{L^2(\Omega)}\quad \foralls v_\ell \in V_\ell.
\end{equation}
By fixing a basis for the space $V_\ell$, we can rewrite \eqref{eq:probDisc} in matrix-vector notation as
\begin{equation}
  \label{eq:probMat}
  	(\beta \mathcal{M}_\ell+\mathcal{B}_\ell)\uv{\ell}
  	= \underline{f}_{\ell},
\end{equation}
where $\mathcal{B}_\ell$ is the biharmonic stiffness matrix, $\mathcal{M}_\ell$ is the mass matrix, $\uv{\ell}$ is the vector representation of the corresponding function $u_\ell$ with respect to the chosen basis and the vector $\underline{f}_{\ell}$ is obtained by testing the right-hand side functional $(f, \cdot)_{L^2(\Omega)}$ with the basis functions.

\begin{notation}
  \label{notation:c}
  Throughout this paper, $c$ is a generic positive constant that is
  independent
  of $h$ and $p$, but may depend on $d$, the constants
  $c_1$, $c_2$, $c_q$, and $c_r$ and the shape of $\Omega$.
\end{notation}

For any two square matrices $A,B\in\mathbb R^{n\times n}$,
$A\le B$ means that
\[
	\underline x^T A \underline x
	\leq
	\underline x^T B \underline x \quad \foralls 
	\underline x\in\mathbb{R}^n.
\]

\section{The multigrid solver}
\label{sec:MG}
In this section, we present an abstract multigrid method and give a convergence theorem
that is based on the analysis by Bramble et al., see
\cite[Theorem 1]{bramble1991convergence}. 

\subsection{The multigrid framework}
Let us assume that we have nested spaces $V_0\subset V_1 \subset \cdots \subset V_L\subset V$. Let $I^{\ell}_{\ell-1}$ be the matrix representation of the canonical embedding from $V_{\ell-1}$ into $V_{\ell}$ and let the
restriction matrix $I^{\ell-1}_{\ell}$ be its transpose, this is
$I^{\ell-1}_{\ell} := (I^{\ell}_{\ell-1})^T$.

On each grid level, $\ell=0,\ldots,L$, we have a linear system
\[
		\mathcal A_\ell \, \underline u_\ell = \underline f_\ell,
\]
which is obtained by discretizing a symmetric, bounded and coercive bilinear form $a(\cdot,\cdot)$ in the space $V_\ell$ using the
Galerkin principle. The matrix induces a norm via
$\|\uv{\ell}\|_{\mathcal{A}_\ell} :=(\mathcal A_\ell \uv{\ell},\uv{\ell})^{1/2} = \|\mathcal A_\ell^{1/2} \uv{\ell}\|$. Here and in what follows, $(\cdot,\cdot)$ and $\|\cdot\|$ are the Euclidean scalar product and norm, respectively. In the continuous setting, the matrix can be represented by an operator
\[
		\mathcal{A}:V\rightarrow V'
		\quad\text{with}\quad
		\mathcal A u = a(u,\cdot).
\]
We have $\|u_{\ell}\|_{\mathcal{A}} = \|\uv{\ell}\|_{\mathcal{A}_\ell}$ for all functions $u_\ell\in V_\ell$ with coefficient representation $\uv{\ell}$.

For the analysis, we can additionally choose symmetric positive definite
matrices $X_\ell$ for all grid levels $\ell=0,1,\ldots,L$, which induce
norms via $\|\uv{\ell}\|_{X_{\ell}}  = (X_{\ell} \uv{\ell},\uv{\ell})^{1/2}  = \|X^{1/2}_{\ell}\uv{\ell}\|$. The norm $\|u_{\ell}\|_{X_{\ell}}$ of a function $u_{\ell}\in V_\ell$ is interpreted as $\|\uv{\ell}\|_{X_{\ell}}$, where $\uv{\ell}$ is the coefficient representation of $u_\ell$. 

For the abstract framework, we assume to have a symmetric and positive definite matrix $\Smo{\ell}$ for every grid level $\ell=1,\ldots,L$, representing the smoother.

Later, for the model problem, the bilinear form $a(\cdot,\cdot)$, 
the matrices $\mathcal{A}_\ell$, $\ell=0,\ldots,L$ and our
choice of ${X}_\ell$ will be
\[
		a(u,v) = \beta (u,v)_{L^2(\Omega)} + (u,v)_{\mathcal B},
	  \quad
		\mathcal{A}_\ell=\beta\mathcal{M}_\ell+\mathcal{B}_\ell
		\quad\text{and}\quad
		X_\ell=(\beta+h_\ell^{-4})\mathcal{M}_\ell+\mathcal{B}_\ell.
\]
As smoothers, we will choose a subspace corrected mass smoother, a
symmetric Gauss-Seidel smoother and a hybrid smoother in
Section~\ref{sec:smoothers}.

Based on these choices, the overall algorithm reads as follows.
\begin{algorithms}
  \label{algo:S}
One multigrid cycle, applied to some iterate $\underline u_\ell^{(0)}$
and a right-hand side $\underline f_\ell$ consists of the following steps:
\begin{itemize}
		\item Apply $\nu_\ell$ pre-smoothing steps, i.e., compute
		\begin{equation}
      \label{eq:algo:Smooth}
      \underline{u}_\ell^{(i)} = \underline{u}_\ell^{(i-1)}
      + \Smo{\ell}(\underline{f}_{\ell}-\mathcal{A}_\ell\underline{u}_\ell^{(i-1)})
      \quad \mbox{for} \quad i=1,\ldots,\nu_\ell.
    \end{equation}
    \item Apply recursive coarse-grid correction, i.e.,
    apply the following steps. Compute the residual and
    restrict it to the next coarser grid level:
    \[
    		\underline r_{\ell-1}
    		=
    		I_\ell^{\ell-1} 
    		(\underline{f}_{\ell}-\mathcal{A}_\ell\underline{u}_\ell^{(\nu_\ell)}).
    \]
    If $\ell-1=0$, compute the update $\underline q_0:= A_0^{-1}
    \underline r_0$ using a direct solver. Otherwise, compute the
    update $\underline q_{\ell-1}$ by applying the algorithm $r$
    ($r\in\mathbb N:=\{1,2,\ldots\}$)
    times recursively to the right-hand side $\underline r_{\ell-1}$
    and a zero vector as initial guess. Then set
    \[
    		\underline u^{(\nu_\ell+1)}_\ell =
    		\underline u^{(\nu_\ell)}_\ell + I_{\ell-1}^\ell
    		\underline q_{\ell-1}.
    \]
    \item Apply $\nu_\ell$ post-smoothing steps, i.e., compute
    $\underline u_\ell^{(i)}$ using~\eqref{eq:algo:Smooth}
    for $i=\nu_\ell+2,\ldots,2\nu_\ell+1$ to obtain the next iterate
    $\underline u_\ell^{(2\nu_\ell+1)}$.
\end{itemize}
\end{algorithms}
This abstract algorithm coincides with the algorithm presented in~\cite{bramble1991convergence}.
Since each multigrid cycle is linear, its application 
can be expressed by the matrix $B_\ell^s$, which is recursively
given by $B_0^s := \mathcal A_0^{-1}$ and
\[
		B_\ell^s  :=
		\big(
		I-(I-\tau_\ell L_\ell^{-1}\mathcal A_\ell)^{\nu_\ell}
		(I-I_{\ell-1}^\ell B_{\ell-1}^s I_\ell^{\ell-1} \mathcal A_\ell)^r
		(I-\tau_\ell L_\ell^{-1}\mathcal A_\ell)^{\nu_\ell}
		\big)\mathcal A_\ell^{-1},
		\quad \ell=1,\ldots,L.
\]
The iteration matrix corresponding to one multigrid cycle is given
by
\[
		I-B_\ell^s \mathcal A_\ell
		=
		(I-\tau_\ell L_\ell^{-1}\mathcal A_\ell)^{\nu_\ell}
		(I-I_{\ell-1}^\ell B_{\ell-1}^s I_\ell^{\ell-1} \mathcal A_\ell)^r
		(I-\tau_\ell L_\ell^{-1}\mathcal A_\ell)^{\nu_\ell},
		\quad\ell=1,\ldots,L
		.
\]
\begin{remark}
  The integer $r$ represents the recursively of the algorithm, where $r=1$ corresponds to the $V$-cycle and $r=2$ corresponds to the $W$-cycle.
\end{remark}
\subsection{Abstract convergence framework}
The assumptions used to show convergence can be split into two groups: \textit{approximation properties} and \textit{smoother properties}. 
\begin{theorem}\label{thrm:abstract}
	Let $\lambda_\ell$ be the largest eigenvalue
	of $X^{-1}_\ell\mathcal{A}_\ell$. Assume that the following
	estimates hold:
	\begin{itemize}
			\item \emph{Approximation properties.} There are constants
			$C_1$ and $C_2$, independent of $\ell$, and linear operators
			$Q_\ell: V_L \rightarrow V_\ell$ for $\ell = 0,1,\ldots,L$
			with $Q_L=I$ such that
			\begin{align}
			\label{eq:ass:approx1}
			  \|(Q_{\ell} -Q_{\ell-1})u_{L}\|^2_{X_\ell} &\leq C_1 \lambda_{\ell}^{-1} (u_{L},u_{L})_{\mathcal{A}} \quad &\text{for}\quad& \ell = 1,\ldots,L,\\
			\label{eq:ass:approx2}
			  (Q_{\ell}u_{L},Q_{\ell}u_{L})_{\mathcal{A}} &\leq C_2  (u_{L},u_{L})_{\mathcal{A}} \quad &\text{for}\quad& \ell = 0,\ldots,L-1,
			\end{align} 
			for all $u_{L} \in V_L$.			
			\item \emph{Smother properties.} We assume there exist
				a constant $C_S$ independent of $\ell$ such that 
		  \begin{equation}
  			  \label{eq:B34}
  				\frac{\|\uv{\ell}\|^2_{X_\ell}}{\lambda_{\ell}} \leq C_S  (\Smo{\ell}X_\ell \uv{\ell},\uv{\ell})_{X_\ell} \quad \foralls \uv{\ell} \in \mathbb{R}^{\dim V_{\ell}}
			\end{equation}
			and
			\begin{equation}
      \label{eq:PositiveSmo}
  (\Smo{\ell}\mathcal{A}_\ell \uv{\ell},\uv{\ell})_{\mathcal{A}_\ell} \leq   ( \uv{\ell},\uv{\ell})_{\mathcal{A}_\ell} \quad \foralls  \uv{\ell}\in\mathbb{R}^{\dim V_\ell}
  		\end{equation}
			holds for $\ell=1,\ldots,L$.
	\end{itemize}
	Then, the estimate
  \[
  \left((I-B^s_L\mathcal{A}_L)\uv{L},\uv{L}\right)_{\mathcal{A}_L}
  \leq \left(1-\frac{1}{CL}\right)
  \left(\uv{L},\uv{L}\right)_{\mathcal{A}_L},
  \]
  holds for all $\uv L \in \mathbb R^{\dim V_L}$,
  where $C = [1+C_2^{1/2}+(C_SC_1)^{1/2}]^{2}$.
\end{theorem}
For a proof, see~\cite[Theorem~1]{bramble1991convergence}.
\begin{remark}
  Condition \eqref{eq:B34} is only required for functions $u_\ell$
  in the range of $Q_{\ell} - Q_{\ell-1}$. However, since we do
  not exploit this, we have stated the stronger condition.
\end{remark}

Now, we provide
conditions that guarantee \eqref{eq:B34} and \eqref{eq:PositiveSmo},
which fit our needs better than the original conditions.
\begin{lemma}\label{lem:smo1}
  If there exists a constant $C_S$, independent of $\ell$, which satisfies
  \begin{equation}
    \label{eq:smo1}
  (\mathcal{A}_\ell\uv{\ell},\uv{\ell}) \leq\frac{1}{\tau_{\ell}}(L_{\ell}\uv{\ell},\uv{\ell}) \leq \lambda_\ell C_S (X_{\ell}\uv{\ell},\uv{\ell}) \quad \foralls \uv{\ell}\in \mathbb{R}^{\dim V_{\ell}} 
  \end{equation}
  for each $\ell = 1,\ldots, L$. Then, the
  assumptions~\eqref{eq:B34} and \eqref{eq:PositiveSmo} hold for the same $C_S$.
  \end{lemma}
\begin{proof}
  We start by showing that the first inequality implies \eqref{eq:PositiveSmo}, i.e., the smoothing operator $I-\Smo{\ell} \mathcal{A}_\ell$ is nonnegative in $\mathcal{A}_\ell$.
  Let $\wv{\ell}\in \mathbb{R}^{\dim V_\ell}$ be an arbitrary vector. Using the Cauchy-Schwarz inequality and the first inequality in \eqref{eq:smo1}, we obtain
  \begin{align*}
    \tau_\ell(L^{-1}_\ell\wv{\ell},\wv{\ell}) &= \tau_\ell(\mathcal{A}^{1/2}_\ell L^{-1}_\ell\wv{\ell},\mathcal{A}^{-1/2}_\ell\wv{\ell})\\
    &\leq \tau_\ell(\mathcal{A}_\ell L^{-1}_\ell\wv{\ell},L^{-1}_\ell\wv{\ell})^{1/2}(\mathcal{A}^{-1}_\ell \wv{\ell},\wv{\ell})^{1/2} \\
    &\leq \tau^{1/2}_\ell(L^{-1}_\ell\wv{\ell},\wv{\ell})^{1/2}(\mathcal{A}^{-1}_\ell \wv{\ell},\wv{\ell})^{1/2}
  \end{align*}
  It follows that
  \[
  \tau_\ell(L^{-1}_\ell\wv{\ell},\wv{\ell}) \leq (\mathcal{A}^{-1}_\ell \wv{\ell},\wv{\ell})\quad \foralls \wv{\ell}\in \mathbb{R}^{\dim V_\ell}.
  \]
  By substituting $\wv{\ell}$ with $\mathcal{A}_\ell\uv{\ell}$, we get \eqref{eq:PositiveSmo}.
  Next, we use the Cauchy-Schwarz inequality and the second inequality in \eqref{eq:smo1} to show \eqref{eq:B34}.
  Let $\wv{\ell}\in \mathbb{R}^{\dim V_{\ell}}$, we have
  \begin{align*}
    (X^{-1}_{\ell}\wv{\ell},\wv{\ell}) &= (L_{\ell}^{1/2}X^{-1}_{\ell}\wv{\ell},L^{-1/2}_{\ell}\wv{\ell}) \leq
    (L_{\ell}X^{-1}_{\ell}\wv{\ell},X^{-1}_{\ell}\wv{\ell})^{1/2}(L^{-1}_{\ell}\wv{\ell},\wv{\ell})^{1/2}\\
    &\leq \tau^{1/2}_{\ell}\lambda_\ell^{1/2} {C}_S^{1/2}(X^{-1}_{\ell}\wv{\ell},\wv{\ell})^{1/2}(L^{-1}_{\ell}\wv{\ell},\wv{\ell})^{1/2}.
  \end{align*}
  By squaring the inequality, we get
  \[
  (X_{\ell}^{-1}\wv{\ell},\wv{\ell}) = \tau_\ell \lambda_\ell C_S(L_{\ell}^{-1} \wv{\ell},\wv{\ell})
  \quad \foralls \wv{\ell}\in \mathbb{R}^{\dim V_\ell}.
  \]
  By substituting $\wv{\ell}$ with $X_{\ell}\uv{\ell}$, we get~\eqref{eq:B34}.
\end{proof}

\section{Approximation error estimates}
\label{Approx}
In this section, we prove some approximation error estimates
and provide a projector which will be used to prove~\eqref{eq:ass:approx1}
and~\eqref{eq:ass:approx2}. 

\subsection{Error and stability estimates for the univariate case}
We start by introducing a periodic spline space.
For any given sequence of grid points $\bm\tau=(0,\tau_1,\ldots,\tau_N,1)$, we define
\[
	\bm{\tau}^{per} := (-1,-\tau_N,\cdots,-\tau_1,0,\tau_1,\cdots,\tau_N,1).
\]
For each $p\in \mathbb N$, we define the periodic spline space 
\begin{equation}
  \nonumber
S_{p,\bm{\tau}}^{per} := \left\lbrace v \in S_{p,\bm{\tau}^{per}}  \,:\, \partial^{l} v \left(-1\right) = \partial^{l}v\left( 1\right)\quad \foralls l\in\mathbb N_0 \mbox{ with } l<p \right\rbrace
\end{equation}
and a spline space with vanishing even derivatives on the boundary
\begin{equation}
  \label{eq:defSEV}
S^{0}_{p,\bm{\tau}} := \left\lbrace v \in S_{p,\bm\tau}  \,:\, \partial^{2l} v \left(0\right) = \partial^{2l}v\left(0 \right) = 0\quad \foralls l\in\mathbb N_0 \mbox{ with } 2l<p \right\rbrace.
\end{equation}
We also define the periodic Sobolev space
\begin{equation}
  \nonumber
H^{q}_{per}(-1,1) := \left\lbrace v \in H^q(-1,1)  \,:\, \partial^{l}v\left(-1\right)=\partial^{l} v\left( 1\right), \quad \foralls l\in\mathbb N_0 \mbox{ with } l<q \right\rbrace
\end{equation}
for each $q\in \mathbb N$.
  Let $\Pi^{per}_{p,\bm{\tau}}:H^2_{per}(-1,1) \rightarrow S_{p,\bm{\tau}}^{per}$ be the $H^2$-orthogonal projector satisfying
  \begin{align}
    \label{eq:perProj}
    \begin{split}
    \inner{\partial^2\Pi^{per}_{p,\bm{\tau}}u}{\partial^2 v}_{L^2(-1,1)} &= \inner{\partial^2 u}{\partial^2 v}_{L^2(-1,1)} \quad \foralls v\in S_{p,\bm{\tau}}^{per}, \\
    \inner{\Pi^{per}_{p,\bm{\tau}}u}{1}_{L^2(-1,1)} &= \inner{u}{1}_{L^2(-1,1)}.
    \end{split}
  \end{align}
We use the following approximation error estimate for spline spaces which does not require uniform knot spans.
\begin{theorem}
  \label{theo:espen4}
  For any $p \geq 3$, we have
  \begin{equation}
  \nonumber
    \|\partial^2(u-\Pi^{per}_{p,\bm{\tau}} u) \|_{L^2(-1,1)} \leq \frac{\hmax_{\bm{\tau}}^2}{\pi^2}  \|\partial^4 u \|_{L^2(-1,1)} \quad \foralls u \in H^4_{per}(-1,1).
  \end{equation}
\end{theorem}
For a proof, see \cite[Theorem 4]{sande2019sharp}.

Using the $H^2$--$H^4$ result above and an Aubin-Nitsche duality trick,
we obtain the following $L^2$--$H^2$ result.
\begin{theorem}
  \label{theo:uniPerL2H2}
  For any $p \geq 3$, we have
  \begin{equation}
   \nonumber
    \|u-\Pi^{per}_{p,\bm{\tau}} u\|_{L^2(-1,1)} \leq \frac{\hmax_{\bm{\tau}}^2}{\pi^2}  \|\partial^2 u \|_{L^2(-1,1)} \quad \foralls u \in H^2_{per}(-1,1).
  \end{equation}
\end{theorem}
\begin{proof}
  Let $u \in H^2_{per}(-1,1)$ be arbitrary but fixed. Let $w\in H^4(-1,1)\cap H^3_{per}(-1,1)$
  be such that $\partial^4 w = u - \Pi^{per}_{p,\bm{\tau}} u$. 
  Note that~\eqref{eq:perProj} gives $0=(u - \Pi^{per}_{p,\bm{\tau}} u,1)_{L^2(-1,1)}
  = (\partial^4 w,1)_{L^2(-1,1)} = \partial^3 w(1) - \partial^3 w(-1)$. So, we know that
  $w\in H^4_{per}(-1,1)$.
  
  Using integration by parts (which does not introduce boundary terms since
  $u - \Pi^{per}_{p,\bm{\tau}} u \in H^2_{per}(-1,1)$ and $w\in H^4_{per}(-1,1)$)
  and using Theorem~\ref{theo:espen4}, we obtain
  \begin{align*}
    \|u - \Pi^{per}_{p,\bm{\tau}} u\|^2_{L^2} &= \frac{(u - \Pi^{per}_{p,\bm{\tau}} u, u - \Pi^{per}_{p,\bm{\tau}} u)_{L^2}}{\|u - \Pi^{per}_{p,\bm{\tau}} u\|_{L^2}}
    =\frac{(u - \Pi^{per}_{p,\bm{\tau}} u, \partial^4 w)_{L^2}}{\|\partial^4 w\|_{L^2}}\\
    &=\frac{(\partial^2(u - \Pi^{per}_{p,\bm{\tau}} u), \partial^2 w)_{L^2}}{\|\partial^4 w\|_{L^2}} \leq \frac{\hmax_{\bm{\tau}}^2}{\pi^2} \frac{(\partial^2(u - \Pi^{per}_{p,\bm{\tau}} u), \partial^2 w)_{L^2}}{\|\partial^2 (w - \Pi^{per}_{p,\bm{\tau}} w)\|_{L^2}}.
  \end{align*} 
  From the definition of $\Pi^{per}_{p,\bm{\tau}}$, see \eqref{eq:perProj}, we have $(\partial^2(u - \Pi^{per}_{p,\bm{\tau}} u), \partial^2 \Pi^{per}_{p,\bm{\tau}} w)_{L^2} = 0$. This, together with the Cauchy-Schwarz inequality and the $H^2$-stability of $\Pi^{per}_{p,\bm{\tau}}$, gives
\begin{align*} 
  \|u - \Pi^{per}_{p,\bm{\tau}} u\|^2_{L^2} &\leq \frac{\hmax_{\bm{\tau}}^2}{\pi^2} \frac{(\partial^2(u - \Pi^{per}_{p,\bm{\tau}} u), \partial^2(w-\Pi^{per}_{p,\bm{\tau}} w))_{L^2}}{\|\partial^2 (w - \Pi^{per}_{p,\bm{\tau}} w)\|_{L^2}} \\
  &\leq \frac{\hmax_{\bm{\tau}}^2}{\pi^2}\|\partial^2 (u - \Pi^{per}_{p,\bm{\tau}} u)\|^2_{L^2} \leq  \frac{\hmax_{\bm{\tau}}^2}{\pi^2}\|\partial^2 u \|^2_{L^2},
\end{align*}
which completes the proof.
\end{proof}

Let $\Pi^{0}_{p,\bm{\tau}}:H^2(0,1)\cap H^1_0(0,1) \rightarrow S_{p,\bm{\tau}}^0$ be the $H^2$-orthogonal projector satisfying
\begin{align}
    \nonumber
    \inner{\partial^2\Pi^{0}_{p,\bm{\tau}}u}{\partial^2 v}_{L^2(0,1)} &= \inner{\partial^2 u}{\partial^2 v}_{L^2(0,1)} \quad \foralls v\in S_{p,\bm{\tau}}^{0}.
\end{align}

\begin{theorem}
  \label{theo:Pi0}
  For any $p \geq 3$, we have
  \begin{equation}
  	\nonumber
    \|u-\Pi^0_{p,\bm{\tau}} u \|_{L^2(0,1)} \leq \frac{\hmax_{\bm{\tau}}^2}{\pi^2}  \|\partial^2 u \|_{L^2(0,1)} \quad \foralls u \in H^2(0,1)\cap H^1_0(0,1).
  \end{equation}
\end{theorem}
\begin{proof}
Let $u\in H^2(0,1) \cap H^1_0(0,1)$ be arbitrary but fixed. Define $w$ on $(-1,1)$ to be
    \[
        w(x) := \mbox{sign}(x)\; u( |x| ).
    \]
    Observe that we obtain $w \in H^2_{per}(-1,1)$. From Theorem~\ref{theo:uniPerL2H2}, we have
    \[
        \| (I-\Pi^{per}_{p,\bm{\tau}}) w \|_{L^2(-1,1)} \le c \hmax_{\bm{\tau}}^2 \|\partial^2 w \|_{L^2(-1,1)}.
    \]
    Observe that $\|\partial^2 w \|_{L^2(-1,1)} = 2^{1/2} \|\partial^2 u \|_{L^2(0,1)}$.  Define $w_{\bm{\tau}}:= \Pi^{per}_{p,\bm{\tau}} w$  and let $u_{\bm{\tau}}$ be the restriction of $w_{\bm{\tau}}$ to $(0,1)$. Observe that $w_{\bm{\tau}}$ is anti-symmetric, which implies that $u_{\bm{\tau}}\in S^{0}_{p,\bm{\tau}}$. It follows that $\|w-w_{\bm{\tau}}\|_{L^2(-1,1)} = 2^{1/2} \|u-u_{\bm{\tau}}\|_{L^2(0,1)}$. Using this, we obtain
    \[
        \| u-u_{\bm{\tau}} \|_{L^2(0,1)} \le c \hmax_{\bm{\tau}}^2 \|\partial^2 u \|_{L^2(0,1)}.
    \]
    It remains to show that $u_{\bm{\tau}}$ coincides with $\Pi^{per}_{p,\bm{\tau}} u$, i.e., to show that $u-u_{\bm{\tau}}$ is $H^2$-orthogonal to $S^{0}_{p,\bm{\tau}}$. By definition, this means that we have to show 
    \[
        (\partial^2(u-u_{\bm{\tau}}),\partial^2 \tilde{u}_{\bm{\tau}})_{L^2(0,1)} = 0 \quad \foralls \tilde{u}_{\bm{\tau}} \in S^{0}_{p,\bm{\tau}}.
    \]
    Let $\tilde{w}_{\bm{\tau}} \in S_{p,\bm{\tau}}^{per}$ be $\tilde{w}_{\bm{\tau}} := \mbox{sign}(x) \,\tilde{u}_{\bm{\tau}}( |x| )$ and observe that $ 2(\partial^2(u-u_{\bm{\tau}}),\partial^2\tilde{u}_{\bm{\tau}})_{L^2(0,1)} = (\partial^2(w-w_{\bm{\tau}}),\partial^2\tilde{w}_{\bm{\tau}})_{L^2(0,1)}$, since $u$, $u_{\bm{\tau}}$ and $\tilde{u}_{\bm{\tau}}$ are restrictions of $w$, $w_{\bm{\tau}}$ and $\tilde{w}_{\bm{\tau}}$, respectively. Furthermore, $(\partial^2(w-w_{\bm{\tau}}),\partial^2\tilde{w}_{\bm{\tau}})_{L^2(-1,1)}=0 $ by construction, since $w_{\bm{\tau}}:= \Pi^{per}_{p,\bm{\tau}} w$, which completes the proof.
\end{proof}
Let $Q^{0}_{p,\bm{\tau}}:H^2(0,1)\cap H^1_0(0,1) \rightarrow S_{p,\bm{\tau}}^0$ be the $L^2$-orthogonal projector satisfying
\begin{align}
    \nonumber
    \inner{Q^{0}_{p,\bm{\tau}}u}{ v}_{L^2(0,1)} &= \inner{ u}{ v}_{L^2(0,1)} \quad \foralls v\in S_{p,\bm{\tau}}^{0}.
\end{align}

Since the $L^2$-orthogonal projector minimizes the error in the $L^2$-norm,
Theorem~\ref{theo:Pi0} immediately implies the following statement.
\begin{theorem}
  \label{theo:Q}
  For any $p \geq 3$, we have
  \begin{equation}
  \nonumber
    \|u-Q^{0}_{p,\bm{\tau}} u \|_{L^2(0,1)} \leq \frac{\hmax_{\bm{\tau}}^2}{\pi^2}  \|\partial^2 u \|_{L^2(0,1)} \quad \foralls u \in H^2(0,1) \cap H^1_0(0,1).
  \end{equation}
\end{theorem}

Next, we show the stability of $Q_{p,\bm{\tau}}^0$ with respect to the $H^2$-seminorm. Such a proof is possible since the space $S_{p,\bm{\tau}}^0$ satisfies the following $p$-robust inverse inequality, while the space $S_{p,\bm{\tau}}\cap H^1_0(0,1)$ does not satisfy such an inverse inequality, cf.~\cite{takacs2016approximation}.
\begin{theorem}
  \label{theo:BiInv2}
  Let $p\in \mathbb{N}$ with $p\geq 2$. We have
  \begin{equation}
   \nonumber
  \|\partial^2 u_{\bm{\tau}}\|_{L^2(0,1)} \leq 12 h^{-2}_{\bm{\tau},\mathrm{min}} \|u_{\bm{\tau}}\|_{L^2(0,1)} \quad \forall u_{\bm{\tau}} \in S^{0}_{p,\bm{\tau}}.
  \end{equation}
\end{theorem}
A proof can be found in \cite[Theorem 12]{sogn2019robust}.
\begin{theorem}
  \label{theo:QHstab1d}
  Let $p\in \mathbb{N}$ with $p\geq 3$. Then there exists a constant $c>0$ such that
    \begin{align*}
    \|\partial^2(Q_{p,\bm{\tau}}^0 u) \|^2_{L^2(0,1)} \leq c\|\partial^2 u\|^2_{L^2(0,1)}
    \quad \forall u \in H^2(0,1)\cap H^1_{0}(0,1).
    \end{align*}
\end{theorem}
\begin{proof}
  The proof is analogous to that of \cite[Theorem 14]{sogn2019robust}, however it is given here for completeness.
  Using the triangle inequality and the inverse inequality, we obtain
  \begin{align*}
  		\|\partial^2 Q_{p,\bm{\tau}}^0 u \|^2_{L^2}
  		& \le
  		2\|\partial^2 \Pi_{p,\bm{\tau}}^0 u \|^2_{L^2}
  		+
  		2\|\partial^2 (Q_{p,\bm{\tau}}^0u - \Pi_{p,\bm{\tau}}^0 u) \|^2_{L^2}
  		\\&\le
  		2\|\partial^2 \Pi_{p,\bm{\tau}}^0 u \|^2_{L^2}
  		+
  		c h_{\bm{\tau},\mathrm{min}}^{-2}
  		\|Q_{p,\bm{\tau}}^0 u - \Pi_{p,\bm{\tau}}^0 u \|^2_{L^2}
  		\\
  		&
  		\le
  		2\|\partial^2 \Pi_{p,\bm{\tau}}^0 u \|^2_{L^2}
  		+
  		c h_{\bm{\tau},\mathrm{min}}^{-2}
  		\|u - \Pi_{p,\bm{\tau}}^0 u \|^2_{L^2}
  		+
  		c h_{\bm{\tau},\mathrm{min}}^{-2}
  		\|u-Q_{p,\bm{\tau}}^0 u \|^2_{L^2}.
  \end{align*}
  The Theorems~\ref{theo:Pi0} and~\ref{theo:Q} and
  Assumption~\eqref{eq:quasiuniform} 
  give the desired result.
\end{proof}

\subsection{Proof of the approximation properties}
\label{subsec:4:3}

In this subsection, we consider the discretization framework from
Section~\ref{sec:prelims}. We choose
\[
		X_\ell := \mathcal{B}_\ell + (\beta + \hmaxh{\ell}^{-4})\mathcal{M}_\ell,
\]
which corresponds to the norm $\|\cdot\|_{X_\ell}$ that satisfies
\[
  \|u\|^2_{X_\ell} = \|u \|^2_{\mathcal{B}} + (\beta + \hmaxh{\ell}^{-4})\|u\|^2_{L^2(\Omega)}
  \quad \foralls u\in V.
\]

Now, we give a bound for the eigenvalues of $X^{-1}_\ell\mathcal{A}_\ell$.
\begin{lemma}
  \label{lem:eigenvalue}
  Let $\lambda_{\ell}$ with $\ell\geq 1$ be the largest eigenvalue of $X_{\ell}^{-1}\mathcal{A}_{\ell}$. For $p\geq 3$, we have $\lambda_{\ell}\in (\frac{1}{1+c},1)$ for some positive constant $c$. 
\end{lemma}
\begin{proof}
Since $\mathcal{M}_\ell$ is symmetric positive definite and $h_\ell^{-4}>0$, we have $\mathcal{A}_{\ell}< X_{\ell}$, which implies $\lambda_{\ell}<1$.

For the lower bound, we use $V_{\ell-1} \subsetneqq V_\ell$, which implies
that there is some $w_\ell\in V_\ell$ that is $L_2$-orthogonal to $V_{\ell-1} $, that is $(w_{\ell},u_{\ell-1})_{L^2(\Omega)} = 0$ for all $u_{\ell-1}\in V_{\ell-1}$. By combining Theorem~\ref{theo:QHapp} and~\eqref{eq:geoEquiv}, we
obtain
\[
\|w_{\ell}\|_{L^2(\Omega)} = \sup_{u_{\ell-1}\in V_{\ell-1}} \|w_{\ell}-u_{\ell-1}\|_{L^2(\Omega)} \leq c \, \hmaxh{\ell-1}^2 \|w_{\ell}\|_{\mathcal B}.
\]
In matrix-vector notation, this reads as
\[
		\underline w_{\ell}^T
		\mathcal{M}_{\ell}
		\underline w_{\ell}
		\leq
		c\,
		\hmaxh{\ell-1}^4\,
		\underline w_{\ell}^T \mathcal{B}_{\ell}\underline w_{\ell}.
\]
Using~\eqref{eq:assGrids}, we know that there is a constant $c>0$
such that
\begin{align*}
  \underline{w}_{\ell}^\top X_{\ell}\underline{w}_{\ell}
  &=
  \underline{w}_{\ell}^\top \mathcal A_{\ell} \underline{w}_{\ell}
  +
  \hmaxh{\ell}^{-4}
  \underline{w}_{\ell}^\top \mathcal{M}_{\ell} \underline{w}_{\ell}
  <
  (1+c)
  \underline{w}_{\ell}^\top \mathcal A_{\ell} \underline{w}_{\ell},
\end{align*}
which shows $\lambda_\ell > 1/(1+c)$.
\end{proof}
Next, we prove \eqref{eq:ass:approx1} and~\eqref{eq:ass:approx2}. This requires that we choose the projectors $\mathbf{Q}^0_{p,\ell}$, which have to map into the space $V_\ell$. We first define a projector that maps from $\widehat V$ into $\widehat V_\ell$ by tensorization of the univariate projectors:
\[
	\widehat{\mathbf{Q}}^0_{p,\ell}
	:=
	Q^0_{p,\bm{\tau}_{\ell,1}}
	\otimes \cdots \otimes  
	Q^0_{p,\bm{\tau}_{\ell,d}},
\]
where the tensor product is to be understood as in \cite[Section~3.2]{T:2017MPMG}. The next two theorems follow from Theorems~\ref{theo:Q} and~\ref{theo:QHstab1d} by standard arguments.
\begin{theorem}
  \label{theo:QHapp}
  Let $p\in \mathbb{N}$ with $p\geq 3$. Then there exists a constant $c$ such that
  \[
  \|(I-\widehat{\mathbf{Q}}_{p,\ell}^0) \widehat u \|_{L^2(\widehat\Omega)} \leq c \hmax_{\ell}^2 \| \widehat u \|_{\bar{\mathcal B}} \quad \foralls \widehat u\in H^2(\widehat\Omega) \cap H^1_0(\widehat\Omega).
  \]
\end{theorem}
\begin{proof}
	The proof is given for the two-dimensional case.
	We have by definition and using the triangle inequality
  \begin{align*}
  		\|(I-\widehat{\mathbf{Q}}_{p,\ell}^0) \widehat u \|_{L^2(\widehat\Omega)}
  		&=
  		\|(I-Q_{p,\bm{\tau}_{\ell,1}}^0\otimes
  		Q_{p,\bm{\tau}_{\ell,2}}^0) \widehat u \|_{L^2(\widehat\Omega)}
  		\\&\le
  		\|(I-Q_{p,\bm{\tau}_{\ell,1}}^0\otimes I) \widehat u \|_{L^2(\widehat\Omega)}
  		+
  		\|(Q_{p,\bm{\tau}_{\ell,1}}^0\otimes I)(I-I\otimes
  		Q_{p,\bm{\tau}_{\ell,2}}^0) \widehat u \|_{L^2(\widehat\Omega)}.
  \end{align*}
  Using the $L^2$-stability of the $L^2$-projectors, we further
  obtain
  \begin{align*}
  		\|(I-\widehat{\mathbf{Q}}_{p,\ell}^0) \widehat u \|_{L^2(\widehat\Omega)}
  		&\le
  		\|(I-Q_{p,\bm{\tau}_{\ell,1}}^0\otimes I) \widehat u \|_{L^2(\widehat\Omega)}
  		+
  		\|(I-I\otimes Q_{p,\bm{\tau}_{\ell,2}}^0) \widehat u \|_{L^2(\widehat\Omega)}.
  \end{align*}
  The desired result immediately follows from Theorem~\ref{theo:Q}.
	The extension to more dimensions is obvious.
\end{proof}

\begin{theorem}
  \label{theo:QHstab}
  Let $p\in \mathbb{N}$ with $p\geq 3$. Then there exists a constant $c>0$ such that
    \begin{align*}
    \|\widehat{\bm{Q}}_{p,\ell}^0 \widehat u \|^2_{\bar{\mathcal B}} \leq c\|\widehat u\|^2_{\bar{\mathcal B}}
    \quad \forall \widehat u \in H^2(\widehat \Omega)\cap H^1_{0}(\widehat\Omega).
    \end{align*}
\end{theorem}
\begin{proof}
	The proof is given for the two-dimensional case.
	We have by definition and using the triangle inequality
  \begin{align*}
  		\|\widehat{\mathbf{Q}}_{p,\ell}^0 \widehat u \|_{\bar{\mathcal B}}^2
  		&=
  		\|\partial_{x_1}^2 (Q_{p,\bm{\tau}_{\ell,1}}^0\otimes
  		Q_{p,\bm{\tau}_{\ell,2}}^0) \widehat u \|_{L^2(\widehat\Omega)}^2
  		+
  		\|\partial_{x_2}^2 (Q_{p,\bm{\tau}_{\ell,1}}^0\otimes
  		Q_{p,\bm{\tau}_{\ell,2}}^0) \widehat u \|_{L^2(\widehat\Omega)}^2
  		\\&=
  		\|\partial_{x_1}^2
  		(I\otimes Q_{p,\bm{\tau}_{\ell,2}}^0)
  		(Q_{p,\bm{\tau}_{\ell,1}}^0\otimes I)
  		\widehat u \|_{L^2(\widehat\Omega)}^2
  		+
  		\|\partial_{x_2}^2
  		(Q_{p,\bm{\tau}_{\ell,1}}^0\otimes I)
  		(I\otimes Q_{p,\bm{\tau}_{\ell,2}}^0)
  		\widehat u \|_{L^2(\widehat\Omega)}^2.
  \end{align*}
  Using the $L^2$-stability of the $L^2$-projector, we obtain 
  \begin{align*}
  		\|\widehat{\mathbf{Q}}_{p,\ell}^0 \widehat u \|_{\bar{\mathcal B}}^2
  		&\le
  		\|\partial_{x_1}^2
  		(I\otimes Q_{p,\bm{\tau}_{\ell,1}}^0)
  		\widehat u \|_{L^2(\widehat\Omega)}^2
  		+
  		\|\partial_{x_2}^2
  		(Q_{p,\bm{\tau}_{\ell,2}}^0\otimes I)
  		\widehat u \|_{L^2(\widehat\Omega)}^2.
  \end{align*}
	Using Theorem~\ref{theo:QHstab1d}, we further obtain
	\begin{align*}
  		\|\widehat{\mathbf{Q}}_{p,\ell}^0 \widehat u \|_{\bar{\mathcal B}}^2
  		&\le
  		c\|\partial_{x_1}^2 \widehat u \|_{L^2(\widehat\Omega)}^2
  		+
  		c\|\partial_{x_2}^2 \widehat u \|_{L^2(\widehat\Omega)}^2
  		= c \|\widehat u\|_{\bar{\mathcal B}}^2,
  \end{align*}
	which finishes the proof.
	The extension to more dimensions is obvious.
\end{proof}

The projectors $\mathbf{Q}^0_{p,\ell}$ are now defined 
via the pull-back principle, such that
\begin{equation}\label{def:Qphys}
			\mathbf{Q}^0_{p,\ell} u
			:=
			(\widehat{\mathbf{Q}}^0_{p,\ell} (u \circ \bm{G}))
			\circ \bm{G}^{-1}
			\quad\foralls u\in V.
\end{equation}
Note that, by construction, $\mathbf{Q}^0_{p,\ell}$ maps into a subspace
of $V_\ell$, where all even outer normal derivatives on the boundary
vanish.

\begin{theorem}
    \label{theo:appProof}
  Let $d\in\mathbb{N}$ and $p\in \mathbb{N}$ with $p \ge 3$. For each level $\ell = 0,1,\ldots, L-1$, let $\mathbf{Q}^0_{p,\ell}:H^2(\Omega)\cap H^1_{0}(\Omega)\rightarrow V_\ell$ be the projectors defined in \eqref{def:Qphys}.
  There exists a constants $C_1$ and $C_2$ such that 
  \begin{align}
    \label{eq:As11}
    \|(\mathbf{Q}^0_{p,\ell}-\mathbf{Q}^0_{p,\ell-1})u_L \|^2_{X_\ell} &\leq C_1 \lambda^{-1}_\ell (u_L,u_L)_{\mathcal{A}} \quad &\text{for}\quad& \ell = 1,\ldots,L,\\
    \label{eq:As12}
(\mathbf{Q}^0_{p,\ell}\,u_L,\mathbf{Q}^0_{p,\ell}\,u_L)_{\mathcal{A}} &\leq C_2  (u_L,u_L)_{\mathcal{A}} \quad &\text{for}\quad& \ell = 0,\ldots,L-1,
  \end{align}
  for all $u_L\in V_L$. 
\end{theorem}
\begin{proof}
		Let $u_L\in V_L$ arbitrary but fixed and
		let $\widehat u_L:= u_L \circ\bm{G}\in \widehat V_L$.
    Using \eqref{eq:geoEquiv}, Lemma~\ref{lemma:equivB}
    and Theorem~\ref{theo:QHstab} and the $L^2$-stability of
    $\widehat{\mathbf{Q}}^0_{p,\ell}$, we obtain
    \begin{align*}
    	(\mathbf{Q}^0_{p,\ell}\, u_L,
      \mathbf{Q}^0_{p,\ell}\, u_L)_{\mathcal{A}}
      &\le c
      (\widehat{\mathbf{Q}}^0_{p,\ell}\,\widehat u_L,
      \widehat{\mathbf{Q}}^0_{p,\ell}\,\widehat u_L)_{\widehat{\mathcal{A}}}
      = c\beta \|\widehat{\mathbf{Q}}^0_{p,\ell}\widehat u_{L}\|^2_{L^2(\widehat \Omega)} + c\|\widehat{\mathbf{Q}}^0_{p,\ell}\widehat u_{L}\|^2_{\widehat{\mathcal{B}}}\\
      &\leq c \beta  \|\widehat u_L\|^2_{L^2(\widehat \Omega)} + c\|\widehat u_{L}\|^2_{\widehat{\mathcal{B}}}
      \leq c  (\widehat u_L,\widehat u_L)_{\widehat{\mathcal{A}}}
      \leq C_2  (u_L,u_L)_{\mathcal{A}},
    \end{align*}
    which shows~\eqref{eq:As12}.
    Next we prove the auxiliary result
    \begin{equation}
      \label{eq:aux3}
    \|(I-\mathbf{Q}^0_{p,\ell-1})u_L \|^2_{X_\ell} \leq c \lambda^{-1}_\ell (u_L,u_L)_{\mathcal{A}} \quad \text{for}\quad \ell = 1,\ldots,L.
    \end{equation}
    Using~\eqref{eq:geoEquiv},~\eqref{lemma:equivB}, Theorem~\ref{theo:QHstab}, Theorem~\ref{theo:QHapp} and the $L^2$-stability of $\mathbf{Q}^0_{p,\ell-1}$, we get
    \begin{align*}
      \|(I-\mathbf{Q}^0_{p,\ell-1})u_L \|^2_{X_\ell} &= \|(I-\mathbf{Q}^0_{p,\ell-1})u_L \|^2_{\mathcal{B}}+ (\beta+\hmaxh{\ell}^{-4})\|(I-\mathbf{Q}^0_{p,\ell-1})u_L \|^2_{L^2(\Omega)}\\
      &\le c \|(I-\widehat{\mathbf{Q}}^0_{p,\ell-1})\widehat{u}_L \|^2_{\bar{\mathcal{B}}}+ c(\beta+\hmaxh{\ell}^{-4})\|(I-\widehat{\mathbf{Q}}^0_{p,\ell-1})\widehat{u}_L \|^2_{L^2(\widehat{\Omega})}\\
      &\leq c \|\widehat u_L\|^2_{\bar{\mathcal{B}}}
      + c\hmaxh{\ell}^{-4}\hmaxh{\ell-1}^{4}\|\widehat u_L\|^2_{\widehat{\mathcal{B}}} + c\beta \|\widehat u_L\|^2_{L^2(\widehat \Omega)}\\
      &\leq c(1+\hmaxh{\ell}^{-4}\hmaxh{\ell-1}^{4})\|u_L\|^2_{\mathcal{B}} + c\beta \|u_L\|^2_{L^2(\Omega)}.
    \end{align*}
    We use assumption \eqref{eq:assGrids} and Lemma~\ref{lem:eigenvalue} 
    to get \eqref{eq:aux3}. To complete the proof, we use the fact that $\mathbf{Q}^0_{p,\ell-1}\mathbf{Q}^0_{p,\ell} = \mathbf{Q}^0_{p,\ell-1}$, \eqref{eq:aux3} and \eqref{eq:As12}, to obtain
    \begin{align*}
      \|(\mathbf{Q}^0_{p,\ell}-\mathbf{Q}^0_{p,\ell-1})u_L \|^2_{X_\ell} &= \|(I-\mathbf{Q}^0_{p,\ell-1})\mathbf{Q}^0_{p,\ell}u_L \|^2_{X_\ell}
      \leq c \lambda^{-1}_\ell (\mathbf{Q}^0_{p,\ell} u_L, \mathbf{Q}^0_{p,\ell}u_L)_{\mathcal{A}}\\
      &\leq C_1 \lambda^{-1}_\ell (u_L, u_L)_{\mathcal{A}}.
    \end{align*}
    This shows~\eqref{eq:As11} and finishes the proof.
\end{proof}

\begin{remark}
  In \cite[Lemma 9.2]{sogn2018schur}, a similar result to Theorem~\ref{theo:QHapp} is shown. There, the $\mathcal{B}_\ell$-orthogonal projector is considered. That proof only holds for uniform grids. By using an $L^2$-orthogonal projector, we avoid these difficulties. Since the convergence theory by Hackbusch \cite{hackbusch2013multi} requires the error estimates for the $\mathcal{B}_\ell$-orthogonal projector, this motivated us to use the convergence theory by Bramble \cite{bramble2018multigrid}, where this is not the case.
\end{remark}

\section{The smoothers and the overall convergence results}
\label{sec:smoothers}

\subsection{Subspace corrected mass smoother}

We consider the subspace corrected mass smoother, which was originally proposed in \cite{hofreither2016robust} for a second order problem and was one of the first smoothers to produce a multigrid method for IgA which is robust in both the grid size and the spline degree. In \cite{sogn2019robust,sogn2018schur} this smoother was extended to biharmonic problems. The smoother is based around the inverse inequality in Theorem~\ref{theo:BiInv2}, which is independent of the spline degree.

First, we introduce a splitting for the one dimensional case as follows:
\[
	S_{p,\bm{\tau}}\cap H^1_0(0,1) = S^{0}_{p,\bm{\tau}} \oplus S^{1}_{p,\bm{\tau}},
\]
where $S^{0}_{p,\bm{\tau}}$ is as defined in \eqref{eq:defSEV} and $S^{1}_{p,\bm{\tau}}$ is its $L^2$-orthogonal complement in $S_{p,\bm{\tau}}\cap H^1_0(0,1)$. For each of these spaces, we define the corresponding $L^2$-orthogonal projection
\begin{align*}
  Q_{p,\bm{\tau}}^0: H^2(0,1)\cap H^1_0(0,1) \rightarrow S^0_{p,\bm{\tau}},\\
  Q_{p,\bm{\tau}}^1: H^2(0,1)\cap H^1_0(0,1) \rightarrow S^1_{p,\bm{\tau}}.
\end{align*}

The next step, is to extend the splitting to the multivariate case.
Let $\alpha:=(\alpha_1,\ldots,\alpha_d) \in \lbrace 0, 1\rbrace^d$ be a multiindex. The tensor product B-spline space $\widehat V_\ell = S_{p,\bm{\tau}_\ell }\cap H^1_0(\widehat\Omega)$ with
$\bm{\tau}_\ell=(\bm{\tau}_{\ell,1},\ldots,\bm{\tau}_{\ell,d})$ is split into the direct sum of $2^d$ subspaces
\begin{equation}
  \label{eq:subspaces}
\widehat V_\ell
= \bigoplus_{\alpha \in \lbrace 0,1\rbrace^d } S^\alpha_{p,\bm{\tau}_{\ell}}
\quad\text{where}\quad
S^\alpha_{p,\bm{\tau}_\ell} = S^{\alpha_1}_{p,\bm{\tau}_{\ell,1}}\otimes \cdots\otimes S^{\alpha_d}_{p,\bm{\tau}_{\ell,d}}.
\end{equation}
Again, we define $L^2$-orthogonal projectors 
\begin{equation}
  \nonumber
  \widehat{\mathbf{Q}}_{p,\bm{\tau}_\ell}^\alpha:= Q_{p,\bm{\tau}_{\ell,1}}^{\alpha_1}\otimes \cdots \otimes Q_{p,\bm{\tau}_{\ell,d}}^{\alpha_d}: \widehat V \rightarrow S^\alpha_{p,\bm{\tau}_\ell}.
\end{equation}
The projector $\widehat{\mathbf{Q}}_{p,\bm{\tau}_\ell}^0$ from Section~\ref{subsec:4:3} is consistent with this definition, for the choice $\alpha=0$.
Since the splitting is $L^2$-orthogonal, we obviously have the following
result.
\begin{equation}
  \label{eq:L2equal}
  	\widehat u_\ell = \sum_{\alpha\in \lbrace 0,1\rbrace^d} \widehat{\mathbf{Q}}_{p,\bm{\tau}_\ell}^\alpha \widehat u_\ell
  	\quad
  	\mbox{and}
 	\quad 
\|\widehat u_\ell\|^2_{L^2(\widehat\Omega)} = \sum_{\alpha\in \lbrace 0,1\rbrace^d} \|\widehat{\mathbf{Q}}_{p,\bm{\tau}_\ell}^\alpha \widehat u_\ell \|^2_{L^2(\widehat\Omega)} \quad \forall \widehat u_\ell \in \widehat V_\ell.
\end{equation}
The next theorem shows that the splitting is also stable in $H^2$.
\begin{theorem}
  \label{theo:QHstab2}
  Let $p\in \mathbb{N}$ with $p\geq 3$. Then there exists a constant $c>0$ such that
  \begin{align*}
    c^{-1}\|\widehat u_\ell\|^2_{\bar{\mathcal{B}}} \leq \sum_{\alpha\in \lbrace 0,1\rbrace^d} \|\widehat{\mathbf{Q}}_{p,\bm{\tau}_\ell}^\alpha \widehat u_\ell \|^2_{\bar{\mathcal{B}}} \leq c\|\widehat u_\ell\|^2_{\bar{\mathcal{B}}}
    \quad \forall \widehat u_\ell \in \widehat V_\ell.
    \end{align*}
\end{theorem}
\begin{proof}
	Theorem~\ref{theo:QHstab1d} states the stability of $Q^0_{p,\bm{\tau}_\ell}$ in the $H^2$-seminorm. The stability of $Q^1_{p,\bm{\tau}_\ell}$ in the $H^2$-seminorm follows using the triangle inequality. The stability of these statements in the $L^2$-norm is obvious. From these observations, the right inequality follows by arguments that are completely analogous to those of the proof of Theorem~\ref{theo:QHstab}.
	
	The left inequality follows from~\eqref{eq:L2equal} and the triangle inequality.
\end{proof}    

For notational convenience, we restrict the setup of the smoother
to the two dimensional case. For notational convenience, we write
the splitting \eqref{eq:subspaces} as
\[
\widehat V_\ell = S^{00}_{p,\bm{\tau}_\ell} \oplus S^{01}_{p,\bm{\tau}_\ell} \oplus S^{10}_{p,\bm{\tau}_\ell} \oplus S^{11}_{p,\bm{\tau}_\ell},
\quad\text{where}\quad
S^{\alpha_1,\alpha_2}_{p,\bm{\tau}_\ell}
=
S^{\alpha_1}_{p,\bm{\tau}_{\ell,1}}
\otimes
S^{\alpha_2}_{p,\bm{\tau}_{\ell,2}}.
\]
Following the ideas of~\cite{hofreither2016robust,sogn2019robust}, we construct local smoothers $L_\alpha$
for any of the spaces $V_{\ell,\alpha}:= S^{\alpha}_{p,\bm{\tau}_\ell}$.
These local contributions are chosen such that they satisfy the corresponding local condition 
\begin{equation}\nonumber
  \bar{\mathcal{B}}_{\ell,\alpha} +\beta \widehat{\mathcal{M}}_{\ell,\alpha}\le L_{\ell,\alpha} \le c (\bar{\mathcal{B}}_{\ell,\alpha} + (\beta + h^{-4}) \widehat{\mathcal{M}}_{\ell,\alpha}),
\end{equation}
where
\[
\bar{\mathcal{B}}_{\ell,\alpha} := \mathbf{P}_{\ell,\alpha}^T \bar{\mathcal{B}}_\ell \mathbf{P}_{\ell,\alpha}
\qquad\mbox{and}\qquad
\widehat{\mathcal M}_{\ell,\alpha} := \mathbf{P}_{\ell,\alpha}^T \widehat{\mathcal M}_\ell \mathbf{P}_{\ell,\alpha}
\]
and $\mathbf{P}_{\ell,\alpha}$ is the matrix representation of the canonical embedding $V_{\ell,\alpha}\rightarrow V_\ell$.
The canonical embedding has tensor product structure, i.e., $P_{\ell,\alpha_1}\otimes\cdots\otimes P_{\ell,\alpha_d}$,
where the $P_{\ell,\alpha_i}$ are the matrix representations of the corresponding univariate embeddings.
In the two-dimensional case, $\bar{\mathcal{B}}_\ell$ and $\widehat{\mathcal M}_\ell$ have the representation
\begin{equation*}
\bar{\mathcal{B}}_\ell = B \otimes M + M \otimes B
\quad\mbox{and}\quad
\widehat{\mathcal M}_\ell = M \otimes M,
\end{equation*}
where $B$ and $M$ are the corresponding univariate stiffness and mass matrices (not necessarily equal for both spacial directions). For notational convenience, we do not indicate the spacial direction and the grid level for these matrices.
Restricting $\bar{\mathcal{B}}_\ell$ to the subspace $V_{\ell,(\alpha_1,\alpha_2)}$ gives
\begin{equation*}
\bar{\mathcal{B}}_{\ell,(\alpha_1,\alpha_2)} = B_{\alpha_1} \otimes M_{\alpha_2} + M_{\alpha_1} \otimes B_{\alpha_2},
\end{equation*}
where $B_{\alpha_i} = P^T_{\ell,\alpha_i} B P_{\ell,\alpha_i}$ and $M_{\alpha_i} = P^T_{\ell,\alpha_i} M P_{\ell,\alpha_i}$.
We define
\[
\bar{\mathcal{A}}_\ell :=\bar{\mathcal{B}}_\ell + \beta \widehat{\mathcal{M}}_\ell \quad\text{and}\quad
\bar{\mathcal{A}}_{\alpha_1,\alpha_2} := \bar{\mathcal{B}}_{\alpha_1,\alpha_2}+\beta\widehat{\mathcal{M}}_{\alpha_1,\alpha_2}.
\]
The inverse inequality for $S^{0}_{p,\bm{\tau}_{\ell,i}}$ (Theorem \ref{theo:BiInv2}), allows us to estimate
\begin{equation*}
B_0\leq \sigma M_0,
\end{equation*}
where $\sigma =  \sigma_0 h_{\ell,\mathrm{min}}^{-4}$ and $\sigma_0 = 144$. Using this, we define the smoothers $L_{\alpha_1,\alpha_2}$ as follows and
obtain estimates for them as follows:
\begin{equation}\nonumber
\begin{aligned}
\bar{\mathcal{A}}_{00} &\leq (2\sigma  + \beta) M_0\otimes M_0 &=: L_{00} \le c(\bar{\mathcal{A}}_{00}+h^{-4} \widehat{\mathcal{M}}_{00}),\\
\bar{\mathcal{A}}_{01} &\leq   M_0 \otimes\left((\sigma +\beta)M_1 +B_1\right) &=: L_{01} \le c(\bar{\mathcal{A}}_{01}+h^{-4} \widehat{\mathcal{M}}_{01}),\\
\bar{\mathcal{A}}_{10} &\leq  \left(B_1  + (\sigma + \beta)M_1 \right)\otimes M_0 &=: L_{10}\le c(\bar{\mathcal{A}}_{10}+h^{-4} \widehat{\mathcal{M}}_{10}),\\
\bar{\mathcal{A}}_{11} &=  B_1 \otimes M_1  + M_1 \otimes B_1 + \beta M_1 \otimes M_1 &=: L_{11}\le c(\bar{\mathcal{A}}_{11}+h^{-4} \widehat{\mathcal{M}}_{11}).
\end{aligned}
\end{equation}
The extension to three and more dimensions is completely straight-forward (cf.~\cite{hofreither2016robust}).
For each of the subspaces $V_{\ell,\alpha}$, we have defined a symmetric and positive
definite smoother $L_\alpha$. The overall smoother is given by 
\begin{equation*}
L_{\ell}:=\sum_{\alpha\in\{0,1\}^d} (\mathbf{Q}^{D,\alpha})^T L_{\alpha} \mathbf{Q}^{D,\alpha},
\end{equation*}
where $\mathbf{Q}^{D,\alpha}= \widehat{\mathcal{M}}_{\alpha}^{-1}\mathbf{P}^T_{\ell,\alpha}\widehat{\mathcal{M}}_{\ell}$ is the matrix representation of the $L^2$-projection from $V_\ell$ to $V_{\ell,\alpha}$. Completely
analogous to~\cite[Section~5.2]{hofreither2016robust}, we obtain
\begin{equation*}
L_{\ell}^{-1}=\sum_{\alpha\in\{0,1\}^d} \mathbf{P}_{\ell,\alpha} L^{-1}_{\alpha} \mathbf{P}^T_{\ell,\alpha}.
\end{equation*}

\begin{theorem}
  \label{theo:SCMS}
  Let $d\in \mathbb{N}$ and $p\in \mathbb{N}$ with $p\geq 3$. The subspace corrected mass smoother $L_{\ell}$, satisfies \eqref{eq:smo1}, i.e.,
  \[
  (\mathcal{A}_\ell\uv{\ell},\uv{\ell}) \leq\frac{1}{\tau_{\ell}}(L_{\ell}\uv{\ell},\uv{\ell}) \leq C_S\,\lambda_\ell \, ((\mathcal{A}_\ell+h^{-4}\mathcal{M}_\ell)\uv{\ell},\uv{\ell}) \quad \foralls \uv{\ell}\in \mathbb{R}^{\dim V_{\ell}} 
  \]
  for all $\tau\in(0,\tau_0)$, where $\tau_0>0$ is some constant. \end{theorem}
\begin{proof}
  The inequality
  \[
  (\bar{\mathcal{A}}_\ell\uv{\ell},\uv{\ell}) \leq(L_{\ell}\uv{\ell},\uv{\ell}) \leq c ((\bar{\mathcal{A}}_\ell+h^{-4}\widehat{\mathcal{M}}_\ell)\uv{\ell},\uv{\ell})
  \]
  was shown in \cite[Theorem 17]{sogn2019robust} for $\beta = 0$. Note that no part of that proof requires uniform grids. So, the proof can be used almost verbatim also in the context of this paper. Using \eqref{eq:L2equal}, the extension to $\beta>0$ is straight forward. Using this and Lemma~\ref{lemma:equivB}, we get
  \begin{align*}
    (\widehat{\mathcal{A}}_\ell\uv{\ell},\uv{\ell}) &\leq d (\bar{\mathcal{A}}_\ell\uv{\ell},\uv{\ell}) \leq\frac{d}{\tau_{\ell}}(L_{\ell}\uv{\ell},\uv{\ell})
    \leq c ((\widehat{\mathcal{A}}_\ell+h^{-4}\widehat{\mathcal{M}}_\ell)\uv{\ell},\uv{\ell})
  \end{align*}
  for some constant $c>0$.
  Using~\eqref{eq:geoEquiv}, we obtain
  \begin{align*}
    ({\mathcal{A}}_\ell\uv{\ell},\uv{\ell}) &\leq\frac{c_1}{\tau_{\ell}}(L_{\ell}\uv{\ell},\uv{\ell})
    \leq c_2 ((\mathcal{A}_\ell+h^{-4}\mathcal{M}_\ell)\uv{\ell},\uv{\ell})
  \end{align*}
  for some constants $c_1,c_2>0$,
  which finishes the proof since $\lambda_\ell$ is bounded from below
  by a constant (Lemma~\ref{lem:eigenvalue}).
\end{proof}

\begin{corollary}\label{final:SCMS}
		Suppose that we solve the linear system~\eqref{eq:probMat} using a
		multigrid solver as outlined in Section~\ref{sec:MG} and
		using the subspace corrected mass smoother as outlined in
		Section~\ref{sec:smoothers}, then the convergence of the multigrid
		solver is described by the relation
		\begin{equation}\label{eq:finalconvergence}
  \left((I-B^s_L\mathcal{A}_L)\uv{L},\uv{L}\right)_{\mathcal{A}_L}
  \leq \left(1-\frac{1}{CL}\right)
  \left(\uv{L},\uv{L}\right)_{\mathcal{A}_L}		
		,
		\end{equation}
		where the constant $C$ is independent of the grid sizes $h_\ell$,
		the number of levels $L$, the spline degree $p$ and the choice
		of the scaling parameter~$\beta$. It may
		depend on $d$, the constants $c_1$, $c_2$, $c_q$, and $c_r$
		and the shape of $\Omega$, cf. Notation~\ref{notation:c}.
\end{corollary}
\begin{proof}
		We use Theorem~\ref{thrm:abstract}, whose assumptions are shown
		by Theorem~\ref{theo:appProof} and the combination of
		Lemma~\ref{lem:smo1} and Theorem~\ref{theo:SCMS}.
\end{proof}

\begin{remark}
The operator $L^{-1}_{\ell}$ can be applied efficiently because all of the local contributions
$L_{00}$, $L_{01}$ and $L_{10}$ can be inverted efficiently because they are tensor products. For example,
we have $L_{00}^{-1} = \frac{1}{2\sigma+\beta} (M^{-1}_0 \otimes I)(I \otimes M^{-1}_0)$,
where both $M^{-1}_0 \otimes I$ and $I \otimes M^{-1}_0$ can be realized by applying direct solvers for
the univariate mass matrix to several right-hand sides. The operator $L_{11}$ is the sum of two
tensor products. So, it has to be inverted as a whole. However, the dimension of the
corresponding space is so small that the corresponding computational costs are negligible. More details on
how to realize the smoother computationally efficient, are given
in \cite[Section 5]{hofreither2016robust}. There, it is outlined where an efficient realization of the subspace corrected mass smoother is also possible in case of more than two dimensions.
\end{remark}

\subsection{Symmetric Gauss-Seidel smoother and a hybrid smoother}
\label{ssub:hybrid}
The second smoother we consider is a symmetric Gauss-Seidel smoother consisting of one forward sweep and one backward sweep. It can be shown that this smoother
satisfies Condition~\eqref{eq:smo1}, where the constant $C_S$ depends on the spline degree, see \cite{sogn2019robust}. This means that also the overall convergence result~\eqref{eq:finalconvergence} holds, where again $C$ depends on the spline degree. The symmetric Gauss-Seidel smoother works well for domains with a nontrivial geometry transformations, but degenerated for large spline degrees (cf. \cite{gahalaut2013multigrid,hofreither2014spectral}).

Since the symmetric Gauss-Seidel smoother works well for nontrivial geometry transformations and the subspace corrected mass smoother is robust with respect to the spline degree, we combine these smoothers into a hybrid smoother, which was first introduced in \cite{sogn2019robust}. This hybrid smoother consists of one forward Gauss-Seidel sweep, followed by one step of the subspace corrected mass smoother, finally followed by one backward Gauss-Seidel sweep.

\section{Numerical experiments}
\label{sec:numerical}
In this section, we present the results of numerical experiments performed with the proposed algorithm. As computational domains, we first consider the unit square, then we consider the nontrivial geometries displayed in Figures~\ref{fig:domains2d} (two dimensional domain) and \ref{fig:domains3d} (three-dimensional domain). We consider the problem
\begin{align}
   \nonumber
  \begin{split}
    \beta u + \Delta^2 u  &= f \quad \text{in} \quad \Omega,\\
    u  &= g_1 \quad \text{on} \quad \partial\Omega,\\
    \Delta u  &= g_2 \quad \text{on} \quad \partial\Omega,
  \end{split}
\end{align}
where
\begin{align*}
  f(x) = (\beta+d^2\pi^4)\prod^d_{k=1}\sin(\pi x_k),\quad  g_1(x) = \prod^d_{k=1}\sin(\pi x_k),\quad g_2(x) = -d\pi^2\prod^d_{k=1}\sin(\pi x_k).
\end{align*}
The discretization space on the parameter domain is the space of tensor-product B-splines. On the coarsest level ($\ell = 0$), we choose
\begin{equation}
  \label{eq:gridpoints}
  \bm{\tau}_{0,i}=
  (0,\,1/3,\,1/2,\,4/5,\,1),
\end{equation}
for all spacial directions $i=1,\ldots,d$. The discretization on level $\ell$ is obtained by preforming $\ell$ uniform $h$-refinement steps. The spline spaces have maximum continuity and spline degree $p$. We solve the resulting system using the preconditioned conjugate gradient (PCG) with a V-cycle multigrid method with 1 pre and 1 post smoothing step, as preconditioner. A random initial guess is used and the stopping criteria is
\[
\|\underline{r}^{(k)}_L\| \leq 10^{-8}\|\underline{r}^{(0)}_L\|,
\]
where $\underline{r}^{(k)}_L:= \underline{f}_L- \mathcal{A}_L\underline{x}^{(k)}_L$ is the residual at step $k$ and $\|\cdot\|$ denotes the Euclidean norm.
All numerical experiments are implemented using the G+Smo library~\cite{gismoweb}.

\subsection{Numerical experiments on parameter domain}
We start with the unit square as the domain, that is, $\Omega = (0,1)^2$. Note that $g_1(x)=g_2(x)=0$ for this domain. For now, we consider the symmetric Gauss-Seidel smoother and the subspace corrected mass smoother. For both smoothers, we choose $\tau = 1$.
The iteration counts are displayed in Table~\ref{t:Para2Db1} for $\beta = 1$, and in Table~\ref{t:Para2Db1e7} for $\beta = 10^7$.
\begin{table}[ht]
\begin{center} 
 \begin{tabular}{| c || c | c | c | c | c | c | c |}
   \hline
   {$\ell\;\diagdown\; p$}&
   {\quad3\quad} & {\quad4\quad} & {\quad5\quad} & {\quad6\quad} &
   {\quad7\quad} & {\quad8\quad} & {\quad9\quad} \\ \hline \hline
 \multicolumn{8}{|l|}{Symmetric Gauss-Seidel} \\ \hline 
   5    & 10 & 16 & 28 & 45 & 71 & 120 & 210\\  \hline
   6    & 10 & 16 & 27 & 44 & 71 & 119 & 209\\   \hline
   7    & 10 & 16 & 27 & 44 & 72 & 117 & 212\\  \hline
   8    & 11 & 16 & 27 & 45 & 72 & 120 & 221\\  \hline\hline
   \multicolumn{8}{|l|}{Subspace corrected mass smoother, $\sigma^{-1}_0 = 0.02$} \\ \hline e
   5    & 126 & 122 & 114 & 105 &  98 &  93 & 85\\  \hline
   6    & 131 & 129 & 123 & 116 & 110 & 105 & 100\\   \hline
   7    & 132 & 133 & 127 & 121 & 116 & 110 & 106\\   \hline
   8    & 133 & 134 & 130 & 124 & 118 & 114 & 110\\   \hline
 \end{tabular}
 \caption{Iteration counts for 2D parametric domain, $\beta = 1$}
 \label{t:Para2Db1}
\end{center}
\end{table}
\begin{table}[ht]
\begin{center} 
 \begin{tabular}{| c || c | c | c | c | c | c | c |}
   \hline
   {$\ell\;\diagdown\; p$}&
   {\quad3\quad} & {\quad4\quad} & {\quad5\quad} & {\quad6\quad} &
   {\quad7\quad} & {\quad8\quad} & {\quad9\quad} \\ \hline \hline
   \multicolumn{8}{|l|}{Symmetric Gauss-Seidel} \\ \hline
   5    & 10 & 16 & 28 & 45 & 71 & 119 & 211\\  \hline
   6    & 10 & 16 & 27 & 44 & 71 & 118 & 208\\   \hline
   7    & 10 & 16 & 27 & 44 & 72 & 117 & 212\\  \hline
   8    & 11 & 16 & 27 & 45 & 72 & 119 & 221\\   \hline
    \multicolumn{8}{|l|}{Subspace corrected mass smoother, $\sigma^{-1}_0 = 0.02$} \\ \hline 
   5    & 124 & 121& 113& 104& 96 & 92 & 85\\  \hline
   6    & 131 & 129& 123& 116& 110& 105& 99\\   \hline
   7    & 132 & 133 & 127 & 120 & 116 & 110 & 106\\   \hline
   8    & 133 & 134 & 130 & 124 & 116 & 118 & 114\\   \hline
 \end{tabular}
 \caption{Iteration counts for 2D parametric domain, $\beta = 10^7$}
 \label{t:Para2Db1e7}
\end{center}
\end{table}

From the tables, we see that the symmetric Gauss-Seidel smoother preforms well for small spline degrees, but degenerates for larger spline degrees. These results are not surprising since it is known that standard smoothers do not work well for large spline degrees (cf. \cite{gahalaut2013multigrid,hofreither2014spectral}).
Due to Corollary~\ref{final:SCMS}, the multigrid solver with subspace corrected mass smoother is robust with respect to the spline degree. The tables do reflex this. However, the iteration numbers are relatively high. Table~\ref{t:Para2Db1uniform} shows the iteration numbers when using an uniform grid with spacing $1/4$ on the coarsest level ($\ell=0$), rather than the grid \eqref{eq:gridpoints}.
The numbers in Table~\ref{t:Para2Db1uniform} are significantly smaller. This implies that the subspace corrected mass smoother is sensitive to the quasi-uniformity constant $c_q$.
\begin{table}[ht]
\begin{center} 
 \begin{tabular}{| c || c | c | c | c | c | c | c |}
   \hline
   {$\ell\;\diagdown\; p$}&
   {\quad3\quad} & {\quad4\quad} & {\quad5\quad} & {\quad6\quad} &
   {\quad7\quad} & {\quad8\quad} & {\quad9\quad} \\ \hline \hline
   \multicolumn{8}{|l|}{Subspace corrected mass smoother, $\sigma^{-1}_0 = 0.015$} \\ \hline 
   5    & 41 & 40 & 39 & 37 & 35 & 34 & 33\\  \hline
   6    & 41 & 41 & 39 & 37 & 36 & 35 & 34\\   \hline
   7    & 42 & 42 & 40 & 39 & 37 & 35 & 35\\   \hline
   8    & 42 & 42 & 41 & 39 & 37 & 37 & 35\\   \hline
 \end{tabular}
 \caption{Iteration counts for 2D parametric domain with uniform grid, $\beta = 1$}
 \label{t:Para2Db1uniform}
\end{center}
\end{table}

\subsection{Numerical experiments on physical domain}

Now, we consider a domain with a nontrivial geometry transformation as displayed in Figures~\ref{fig:domains2d} and \ref{fig:domains3d}.
The convergence of subspace corrected mass smoother degrades significantly due to the nontrivial geometry mapping. To combat this, we consider the hybrid smoother described in Section~\ref{ssub:hybrid}.
\begin{figure}[h]
  \center
\begin{minipage}{0.47\textwidth}
  \centering \includegraphics[width=0.56\textwidth]{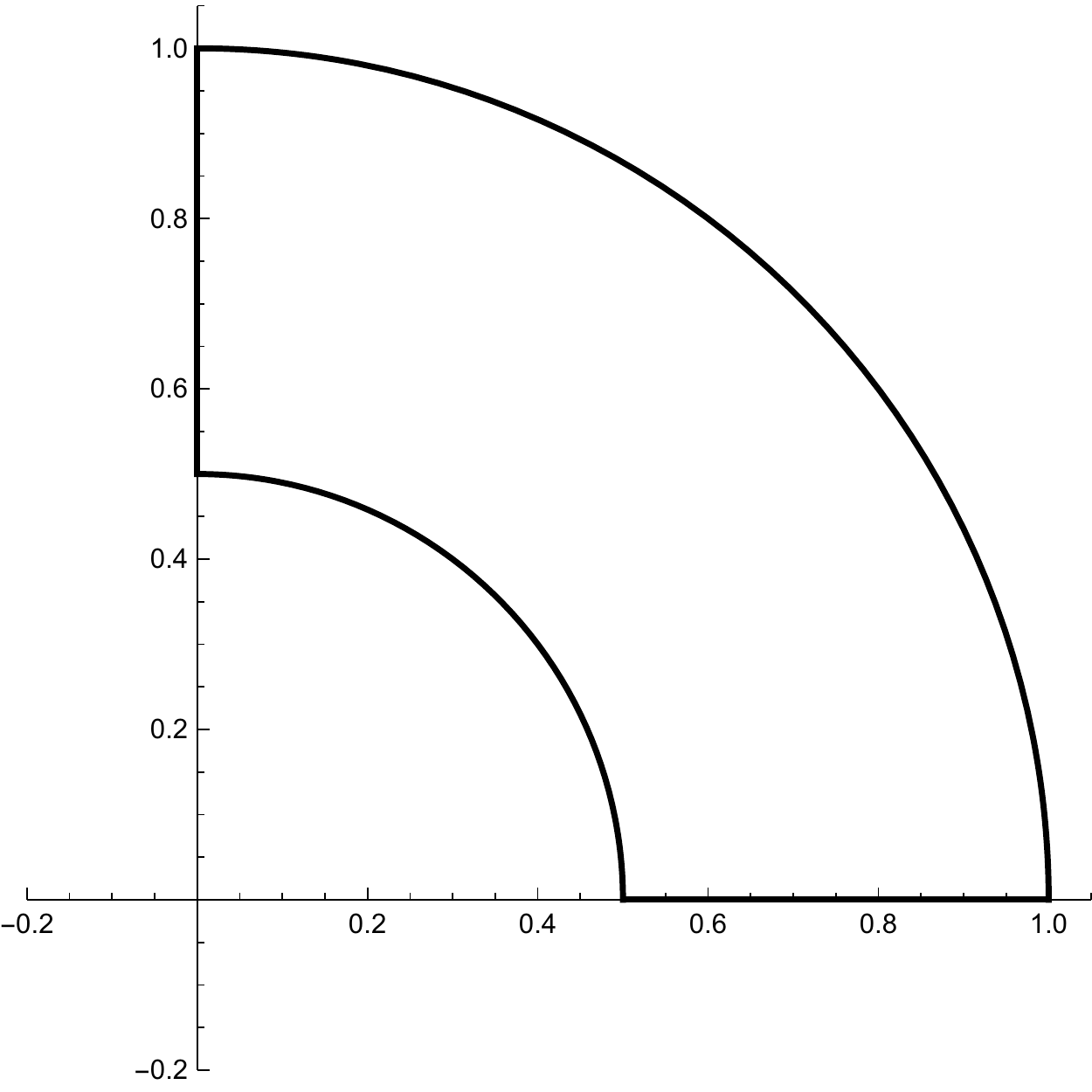}
  \caption{The two-dimensional domain}
  \label{fig:domains2d}
\end{minipage}
\begin{minipage}{0.47\textwidth}
  \centering \includegraphics[width=0.56\textwidth]{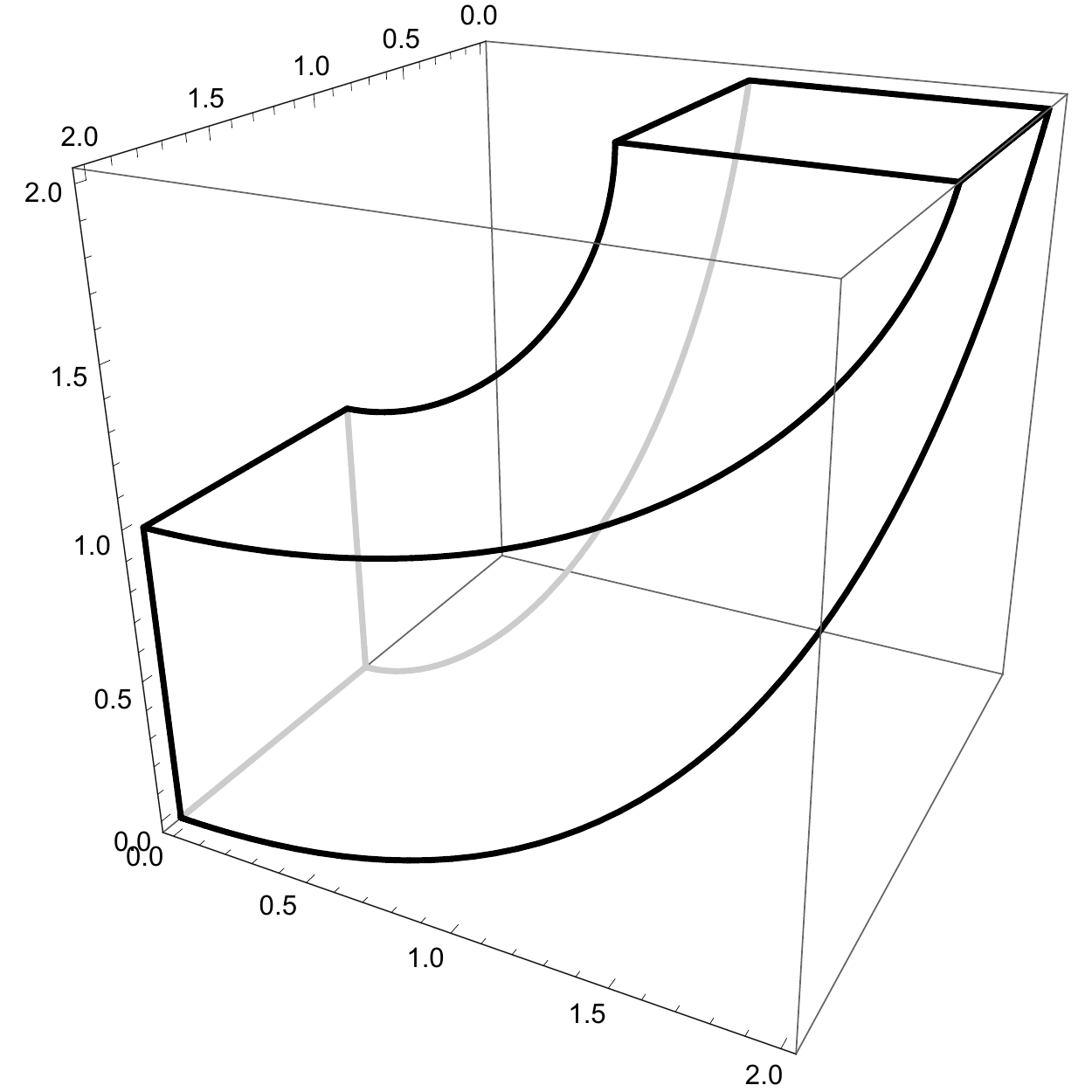}
  \caption{The three-dimensional domain}
  \label{fig:domains3d}
\end{minipage}
\end{figure}
Table~\ref{t:Ann2D} and Table~\ref{t:Ann3D} display the iteration numbers for the 2D and 3D physical domains, respectively. These iteration numbers are relatively small and seam to be robust with respect to both grid size and spline degree. Although the hybrid smoother is more expensive, as one smoothing step can be view as two smoothing steps, the reduction of iteration numbers outweigh this cost for larger spline degrees $p>4$. For smaller spline degrees, the symmetric Gauss-Seidel smoother is ideal choice. 
\begin{table}[ht]
\begin{center} 
 \begin{tabular}{| c || c | c | c | c | c | c | c |}
   \hline
   {$\ell\;\diagdown\; p$}&
   {\quad3\quad} & {\quad4\quad} & {\quad5\quad} & {\quad6\quad} &
   {\quad7\quad} & {\quad8\quad} & {\quad9\quad} \\ \hline \hline
   \multicolumn{8}{|l|}{Hybrid smoother, $\beta = 1$} \\ \hline 
   5    & 28 & 23 & 23 & 24 & 26 & 27 & 27\\  \hline
   6    & 28 & 23 & 22 & 25 & 24 & 26 & 26 \\  \hline
   7    & 29 & 23 & 22 & 23 & 24 & 24 & 24\\   \hline
   8    & 28 & 22 & 21 & 21 & 22 & 22 & 22\\   \hline\hline
   \multicolumn{8}{|l|}{Hybrid smoother, $\beta = 10^7$} \\ \hline 
   5    & 27 & 23 & 23 & 24 & 26 & 27 & 28 \\  \hline
   6    & 28 & 23 & 22 & 25 & 25 & 26 & 26 \\  \hline
   7    & 29 & 23 & 22 & 23 & 24 & 24 & 24\\   \hline
   8    & 28 & 22 & 21 & 21 & 22 & 22 & 22\\   \hline
 \end{tabular}
 \caption{Iteration counts for 2D Physical domain, $\sigma^{-1}_0 = 0.015$, $\tau = 0.1$}
 \label{t:Ann2D}
\end{center}
\end{table}
\begin{table}[ht]
\begin{center} 
 \begin{tabular}{| c || c | c | c | c | c |}
   \hline
   {$\ell\;\diagdown\; p$}&
   {\quad3\quad} & {\quad4\quad} & {\quad5\quad} & {\quad6\quad} &
   {\quad7\quad}  \\ \hline \hline
   \multicolumn{6}{|l|}{Hybrid smoother, $\beta = 1$} \\ \hline 
   1    & 16 & 18 & 21 & 27 & 30 \\  \hline
   2    & 31 & 28 & 26 & 29 & 32 \\  \hline
   3    & 46 & 37 & 33 & 33 & 35 \\  \hline
   4    & 50 & 41 & 34 & 34 & mem\\   \hline\hline
   \multicolumn{6}{|l|}{Hybrid smoother, $\beta = 10^7$} \\ \hline 
   1    & 10 & 11 & 13 & 17 & 20 \\  \hline
   2    & 12 & 16 & 20 & 25 & 29 \\  \hline
   3    & 16 & 19 & 22 & 24 & 28 \\  \hline
   4    & 29 & 28 & 28 & 28 & mem\\   \hline
 \end{tabular}
 \caption{Iteration counts for 3D Physical domain, $\sigma^{-1}_0 = 0.020$, $\tau = 0.1$}
 \label{t:Ann3D}
\end{center}
\end{table}
\begin{remark}
  All experiments have also been performed for the choice $\beta = 0$. In this case, one obtains iteration numbers that are identical than those obtained for $\beta = 1$. Therefore, we chose to only display the results for $\beta = 1$.
\end{remark}

\textbf{Acknowledgements.} This research was funded 
by the Austrian Science Fund (FWF): P31048.

\bibliographystyle{siamplain}
\bibliography{bibliography}

\begin{thebibliography}{10}

\bibitem{beigl2019robust}
{\sc A.~Beigl, J.~Sogn, and W.~Zulehner}, {\em Robust preconditioners for
  multiple saddle point problems and applications to optimal control problems},
  SIAM Journal on Matrix Analysis and Applications, 41 (2020), pp.~1590--1615.

\bibitem{bramble2018multigrid}
{\sc J.~H. Bramble}, {\em Multigrid methods}, Routledge, 2018.

\bibitem{bramble1991convergence}
{\sc J.~H. Bramble, J.~E. Pasciak, J.~P. Wang, and J.~Xu}, {\em Convergence
  estimates for multigrid algorithms without regularity assumptions},
  Mathematics of Computation, 57 (1991), pp.~23--45.

\bibitem{doi:10.1137/0726062}
{\sc S.~C. Brenner}, {\em An optimal-order nonconforming multigrid method for
  the biharmonic equation}, SIAM Journal on Numerical Analysis, 26 (1989),
  pp.~1124--1138.

\bibitem{chen2015multigrid}
{\sc L.~Chen, J.~Hu, and X.~Huang}, {\em Multigrid methods for
  {H}ellan--{H}errmann--{J}ohnson {M}ixed {M}ethod of {K}irchhoff {P}late
  {B}ending {P}roblems}, Journal of Scientific Computing,  (2015), pp.~1--24.

\bibitem{ciarlet2002finite}
{\sc P.~G. Ciarlet}, {\em The finite element method for elliptic problems},
  SIAM, 2002.

\bibitem{de2020robust}
{\sc A.~P. de~la Riva, F.~J. Gaspar, and C.~Rodrigo}, {\em On the robust
  solution of an isogeometric discretization of bilaplacian equation by using
  multigrid methods}, Computers \& Mathematics with Applications, 80 (2020),
  pp.~386--394.

\bibitem{gahalaut2013multigrid}
{\sc K.~P.~S. Gahalaut, J.~K. Kraus, and S.~K. Tomar}, {\em Multigrid methods
  for isogeometric discretization}, Computer methods in applied mechanics and
  engineering, 253 (2013), pp.~413--425.

\bibitem{girault2012finite}
{\sc V.~Girault and P.-A. Raviart}, {\em Finite element methods for
  Navier-Stokes equations: theory and algorithms}, vol.~5, Springer Science \&
  Business Media, 2012.

\bibitem{grisvard1992singularities}
{\sc P.~Grisvard}, {\em Singularities in boundary value problems}, Springer,
  Berlin, 1992.

\bibitem{grisvard2011elliptic}
{\sc P.~{Grisvard}}, {\em {Elliptic Problems in Nonsmooth Domains. Reprint of
  the 1985 hardback ed.}}, Philadelphia, PA: Society for Industrial and Applied
  Mathematics (SIAM), 2011.

\bibitem{hackbusch2013multi}
{\sc W.~Hackbusch}, {\em Multi-Grid Methods and Applications}, vol.~4, Springer
  Science \& Business Media, 2013.

\bibitem{hanisch1993multigrid}
{\sc M.~R. Hanisch}, {\em Multigrid preconditioning for the biharmonic
  {D}irichlet problem}, SIAM Journal on Numerical Analysis, 30 (1993),
  pp.~184--214.

\bibitem{hofreither2016robust}
{\sc C.~Hofreither and S.~Takacs}, {\em Robust multigrid for isogeometric
  analysis based on stable splittings of spline spaces}, SIAM Journal on
  Numerical Analysis, 4 (2017), pp.~2004--2024.

\bibitem{hofreither2017robust}
{\sc C.~Hofreither, S.~Takacs, and W.~Zulehner}, {\em A robust multigrid method
  for isogeometric analysis in two dimensions using boundary correction},
  Computer Methods in Applied Mechanics and Engineering, 316 (2017),
  pp.~22--42.

\bibitem{hofreither2014spectral}
{\sc C.~Hofreither and W.~Zulehner}, {\em Spectral analysis of geometric
  multigrid methods for isogeometric analysis}, in {International Conference on
  Numerical Methods and Applications}, Springer, 2014, pp.~123--129.

\bibitem{gismoweb}
{\sc A.~Mantzaflaris, J.~Sogn, S.~Takacs, and others~(see website)}, {\em
  {G+Smo}}.
\newblock https://github.com/gismo/gismo, 2021.

\bibitem{MarNieNor17}
{\sc K.-A. Mardal, B.~F. Nielsen, and M.~Nordaas}, {\em {Robust preconditioners
  for PDE-constrained optimization with limited observations}}, {BIT}, 57
  (2017), pp.~405--431.

\bibitem{mardal2020robust}
{\sc K.-A. Mardal, J.~Sogn, and S.~Takacs}, {\em Robust preconditioning and
  error estimates for optimal control of the convection-diffusion-reaction
  equation with limited observation in isogeometric analysis}, SIAM Journal on
  Numerical Analysis,  (2022).
\newblock To appear. (Preprint available:
  \url{https://arxiv.org/abs/2012.13003}).

\bibitem{olshanskii2000convergence}
{\sc M.~A. Olshanskii and A.~Reusken}, {\em On the convergence of a multigrid
  method for linear reaction-diffusion problems}, Computing, 65 (2000),
  pp.~193--202.

\bibitem{rafetseder2018decomposition}
{\sc K.~Rafetseder and W.~Zulehner}, {\em A decomposition result for
  {K}irchhoff plate bending problems and a new discretization approach}, SIAM
  Journal on Numerical Analysis, 56 (2018), pp.~1961--1986.

\bibitem{sande2019sharp}
{\sc E.~Sande, C.~Manni, and H.~Speleers}, {\em Sharp error estimates for
  spline approximation: {E}xplicit constants, $n$-widths, and eigenfunction
  convergence}, Mathematical Models and Methods in Applied Sciences, 29 (2019),
  pp.~1175--1205.

\bibitem{sande2020explicit}
{\sc E.~Sande, C.~Manni, and H.~Speleers}, {\em Explicit error estimates for
  spline approximation of arbitrary smoothness in isogeometric analysis},
  Numerische Mathematik,  (2020), pp.~1--41.

\bibitem{sogn2018schur}
{\sc J.~Sogn}, {\em Schur complement preconditioners for multiple saddle point
  problems and applications}, PhD thesis, Johannes Kepler University Linz,
  2018.

\bibitem{sogn2019robust}
{\sc J.~Sogn and S.~Takacs}, {\em Robust multigrid solvers for the biharmonic
  problem in isogeometric analysis}, Computers \& Mathematics with
  Applications, 77 (2019), pp.~105--124.

\bibitem{SogZul18}
{\sc J.~Sogn and W.~Zulehner}, {\em Schur complement preconditioners for
  multiple saddle point problems of block tridiagonal form with application to
  optimization problems}, IMA Journal of Numerical Analysis, 39 (2018),
  pp.~1328--1359.

\bibitem{T:2017MPMG}
{\sc S.~Takacs}, {\em Robust approximation error estimates and multigrid
  solvers for isogeometric multi-patch discretizations}, Mathematical Models
  and Methods in Applied Sciences, 28 (2018), pp.~1899--1928.

\bibitem{takacs2016approximation}
{\sc S.~Takacs and T.~Takacs}, {\em Approximation error estimates and inverse
  inequalities for {B}-splines of maximum smoothness}, Mathematical Models and
  Methods in Applied Sciences, 26 (2016), pp.~1411--1445.

\bibitem{Zhang:Xu}
{\sc S.~Zhang and J.~Xu}, {\em Optimal solvers for fourth-order pdes
  discretized on unstructured grids}, SIAM Journal on Numerical Analysis, 52
  (2014), pp.~282--307.

\end{thebibliography}
\end{document}